\documentclass[12pt]{article}
\usepackage{amssymb}
\usepackage{amsmath}
\usepackage{amsthm}
\usepackage{color}
\usepackage{comment}
\usepackage{geometry}		
\usepackage{graphicx}
\usepackage{hyperref}

\let\originalleft\left
\let\originalright\right
\renewcommand{\left}{\mathopen{}\mathclose\bgroup\originalleft}
\renewcommand{\right}{\aftergroup\egroup\originalright}
\def\cN{\mathcal{N}}
\def\cO{\mathcal{O}}

\newtheorem{theorem}{Theorem}

\newtheorem{lemma}[theorem]{Lemma}

\theoremstyle{definition}
\newtheorem{definition}{Definition}

\theoremstyle{remark}

\begin{document}
\title{
Homoclinic tangencies with infinitely many asymptotically stable single-round periodic solutions.
}
\author{
S.S.~Muni, R.I.~McLachlan, and D.J.W.~Simpson\\\\
School of Fundamental Sciences\\
Massey University\\
Palmerston North\\
New Zealand
}
\maketitle
\begin{abstract}
We consider a homoclinic orbit to a saddle fixed point of an arbitrary $C^\infty$ map $f$ on $\mathbb{R}^2$ and study the phenomenon that $f$ has an infinite family of asymptotically stable, single-round periodic solutions. From classical theory this requires $f$ to have a homoclinic tangency. We show it also necessary for $f$ to satisfy a `global resonance' condition and for the eigenvalues associated with the fixed point, $\lambda$ and $\sigma$, to satisfy $|\lambda \sigma| = 1$. The phenomenon is codimension-three in the case $\lambda \sigma = -1$, but codimension-four in the case $\lambda \sigma = 1$
because here the coefficients of the leading-order resonance terms associated with $f$ at the fixed point must add to zero. We also identify conditions sufficient for the phenomenon to occur, illustrate the results for an abstract family of maps, and show numerically computed basins of attraction.
\end{abstract}

\section{Introduction}
The attractors of a dynamical system govern its typical long-time behaviour. The presence of many attractors is relatively exotic but occurs in diverse applications \cite{Fe08}. Examples include ocean circulation for which several different convection patterns can be stable simultaneously \cite{Ra95b}. In neuroscience the complex nature of information processing and storage of neurons has been attributed to the coexistence of several stable beating and bursting states \cite{CaBa93}. In coupled chemical systems, infinitely many attractors may be possible \cite{NgFe11}.

This paper treats multistability arising from homoclinic tangencies in their simplest setting:
for saddle fixed points of two-dimensional maps.
Such a fixed point has one-dimensional stable and unstable manifolds
and a tangential intersection between these produces a homoclinic tangency, Fig.~\ref{fig:tangency_sketch}.
This is a codimension-one phenomenon because its occurrence can be equated to a single scalar condition.
There exists a vast literature on homoclinic tangencies
as they are a fundamental mechanism for the destruction of hyperbolicity and the creation of robust chaos \cite{PaTa93}.
\begin{figure}[t!]
\begin{center}
\includegraphics{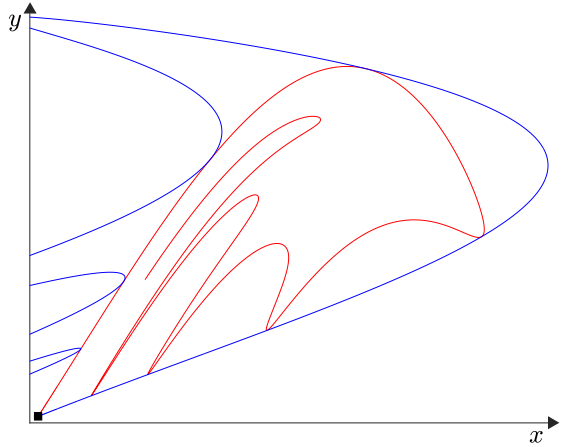}
\caption{A sketch of tangentially intersecting stable [blue] and unstable [red] manifolds of a saddle fixed point of a two-dimensional map. Note that the tangential intersections form a homoclinic orbit.
}
\label{fig:tangency_sketch}
\end{center}
\end{figure}
In particular, Newhouse \cite{Ne74} showed it is typical for homoclinic tangencies to give rise
to the coexistence of infinitely many asymptotically stable periodic solutions when the map is perturbed.
These periodic solutions are multi-round, meaning they
involve more than one excursion far from the fixed point before repeating.
For generic homoclinic tangencies, all single-round periodic solutions (that are sufficiently close to the homoclinic
orbit associated with the tangency) are unstable \cite{GaSi72,GaSi73}.
In this paper we address the following question: in the space of two-dimensional $C^\infty$ maps,
what degeneracy allows for the coexistence of infinitely many asymptotically stable, single-round periodic solutions?

This question has been answered for more restricted classes of maps.
For piecewise-linear continuous maps such infinite coexistence occurs as a codimension-three phenomenon:
the stable and unstable manifolds must coincide on intervals and it is necessary to have $|\lambda \sigma| = 1$,
where $\lambda$ and $\sigma$ are the eigenvalues associated with the fixed point \cite{Si14b,Si14,DSCT}.
For smooth area-preserving maps such infinite coexistence is codimension-two:
in addition to the tangency the map needs to satisfy a `global resonance' condition
(discussed in \S\ref{sec:sec2p2}) 
\cite{DeGo15,DeGo14,GoGo09,GoSh05}.
For $C^\infty$ maps we find that the map needs to satisfy these two conditions and
$|\lambda \sigma| = 1$ (obtained for free in the area-preserving setting).
This suggests the phenomenon is codimension-three, and we show it is codimension-three in the case $\lambda \sigma = -1$
by obtaining sufficient conditions for infinite coexistence.
Interestingly in the case $\lambda \sigma = 1$ the leading-order resonance terms need to satisfy a certain condition
(that holds automatically in the area-preserving setting)
and consequently the phenomenon here is codimension-four.

We can then expect such infinite coexistence to arise as codimension-three and four bifurcations
in prototypical families of maps containing sufficiently many parameters,
such as the generalised H\'enon map of \cite{GoKu05}.
These bifurcations will likely be hubs for multistability,
but to date we are not aware of any examples of this, perhaps due to the high codimension.
Certainly there are several examples of comparable codimension-two phenomena
that serve as focal points of parameter space
and theoretical unfoldings obtained for these are crucial to explaining the greater bifurcation structure,
for instance \cite{Ga94,HiLa95,Si20}.

The remainder of this paper is organised as follows. 
In \S\ref{sec:setup} we state the main results and
in \S\ref{sec:example} we illustrate these with an abstract family of maps.
The next three sections contain proofs.
In \S\ref{sec:necessaryConditions} we obtain three conditions necessary for the infinite coexistence to occur.
Our basic strategy is to estimate $D f^n$ at one point of a single-round periodic solution of period $n$.
The eigenvalues of this matrix determine the stability of the periodic solution
and our results are achieved by determining conditions that prevent the trace of $D f^n$ diverging as $n \to \infty$.
In \S\ref{sec:positiveCase} we establish additional results for the orientation-preserving case $\lambda \sigma = 1$
and in \S\ref{sec:negativeCase} we similarly treat the orientation-reversing case $\lambda \sigma = -1$.
Finally conclusions are presented in \S\ref{sec:conc}.
\label{sec:intro}
\section{Main results}

\label{sec:setup}
\subsection{Local coordinates}
\label{sec:sec2p1}

Let $f$ be a $C^\infty$ map on $\mathbb{R}^2$.
Suppose $f$ has a homoclinic orbit to a saddle fixed point
with eigenvalues $\lambda, \sigma \in \mathbb{R}$ where
\begin{equation}
0 < |\lambda| < 1 < |\sigma|.
\label{eq:eigenvalueAssumption}
\end{equation}
By applying an affine coordinate change
we can assume the fixed point lies at the origin, $(x,y) = (0,0)$,
about which the stable manifold is tangent to the $x$-axis
and the unstable manifold is tangent to the $y$-axis.
Then by Sternberg's linearisation theorem \cite{Qualth,St58}
there exists a locally valid $C^\infty$ coordinate change under which $f$ is transformed to
\begin{equation}
T_0(x,y) = \begin{bmatrix}
\lambda x \left( 1 + x y F(x,y) \right) \\
\sigma y \left( 1 + x y G(x,y) \right)
\end{bmatrix},
\label{eq:T01}
\end{equation}
where $F$ and $G$ are $C^\infty$.
In these new coordinates let $\cN$ be a bounded convex neighbourhood of the origin for which
\begin{equation}
f(x,y) = T_0(x,y), \quad \text{for all}~(x,y) \in \cN,
\label{eq:N}
\end{equation}
see Fig.~\ref{fig:transv_intn_article}. 
\begin{figure}[t!]
\begin{center}
\includegraphics{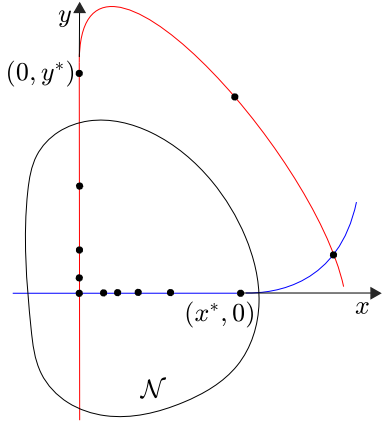}
\caption{A sketch of the stable [blue] and unstable [red] manifolds of the origin for a $C^\infty$ map $f$ satisfying \eqref{eq:N}. A homoclinic orbit is indicated with black dots in the case $0<\lambda<1$ and $\sigma > 1$.}
\label{fig:transv_intn_article}
\end{center}
\end{figure}

The functions $F$ and $G$ represent resonant terms that cannot be eliminated by the coordinate change.
It is possible to make $F$ and $G$ identically zero if $\lambda$ and $\sigma$ are {\em non-resonant}
(in the sense that there do not exist integers $p \ge -1$ and $q \ge -1$ with $p + q \ge 1$ for which $\lambda^p \sigma^q = 1$,
\cite{Qualth,St58}).

In $\cN$ the local stable and unstable manifolds of the origin coincide with the $x$ and $y$-axes, respectively.
Let $(x^*,0)$ and $(0,y^*)$ be some points of the homoclinic orbit.
Without loss of generality we assume $x^* > 0$ and $y^* > 0$.
We further assume $x^*$ and $y^*$ are chosen sufficiently small that
\begin{equation}
(x^*,0), (\lambda x^*,0), \left( 0, \tfrac{y^*}{\sigma} \right), \left( 0, \tfrac{y^*}{\sigma^2} \right) \in \cN.
\label{eq:fourPoints}
\end{equation}

Note that if $\lambda,\sigma > 0$ then the condition that $(\lambda x^{*},0)$ and $\left(0,\frac{y^{*}}{\sigma^{2}}\right)$ belong to $\cN$ are redundant. We do not require $(0,y^{*})$ to belong to $\cN$ because it can be interpreted as the starting point of the excursion so does not need to map under $T_{0}$. Equation $\eqref{eq:fourPoints}$ ensures that the forward orbit of $(x^*,0)$ and a backward orbit of $(0,y^*)$ (these are both part of the homoclinic orbit) converge to the origin in $\cN$. By assumption $f^m(0,y^*) = (x^*,0)$ for some $m \ge 1$.
We let $T_1 = f^m$ and expand $T_{1}$ in a Taylor series centred at $(x,y) = (0,y^*)$:
\begin{equation}
T_1(x,y) = \begin{bmatrix}
x^* + c_1 x + c_2 (y - y^*) + \cO \left( \left( x, y-y^* \right)^2 \right) \\
d_1 x + d_2 (y - y^*) + d_3 x^2 + d_4 x (y - y^*) + d_5 (y - y^*)^2 + \cO \left( \left( x, y-y^* \right)^3 \right)
\end{bmatrix},
\label{eq:T11}
\end{equation}
where here we have explicitly written the terms that will be important below.
Here and throughout this paper $\cO$ denotes big-O notation \cite{De81}.

The value of $m$ and the values of the coefficients $c_i, d_i \in \mathbb{R}$
depend on our choice of $x^*$ and $y^*$.
It is a simple exercise to show that $d_1 = d_1(x^*,y^*) = \frac{\xi y^*}{x^*}$, for some constant $\xi$.
That is, $\frac{d_1 x^*}{y^*}$ is invariant with respect to a change in our choice of $x^*$ and $y^*$. This invariant (analogous to a quantity denoted $\tau$ in  \cite{DeGo15,GoSh87}) will be important below.

From \eqref{eq:T11} we have det$\Big(DT_{1}\big(0,y^{*}\big)\Big) = c_{1}d_{2} - c_{2}d_{1}$. Thus if $c_{1}d_{2} - c_{2}d_{1} \neq 0$ then $f$ is locally invertible along the homoclinic orbit.
The stable and unstable manifolds of the origin intersect tangentially at $(0,y^*)$ if and only if $d_2 = 0$.
From basic homoclinic theory \cite{PaTa93,Ro04}
it is known that a tangential intersection is necessary to have stable single-round periodic solutions
near the homoclinic orbit.
For completeness we prove this as part of Theorem \ref{th:necessaryConditions} below.

\subsection{Three necessary conditions for infinite coexistence}
\label{sec:sec2p2}
It can be expected that points of single-round periodic solutions in $\cN$
converge to the $x$ and $y$-axes as the period tends to infinity.
This motivates our introduction of the set
\begin{equation}
\cN_\eta = \left\{ (x,y) \in \cN \,\big|\, |x y| < \eta \right\},
\label{eq:Nalt}
\end{equation}
where $\eta > 0$.
Below
we control the resonance terms in \eqref{eq:T01} by
choosing the value of $\eta$ to be sufficiently small. We first provide an $\eta$-dependent definition of single-round periodic solutions (SR abbreviates single-round).
\begin{figure}\hspace{3cm}
\includegraphics{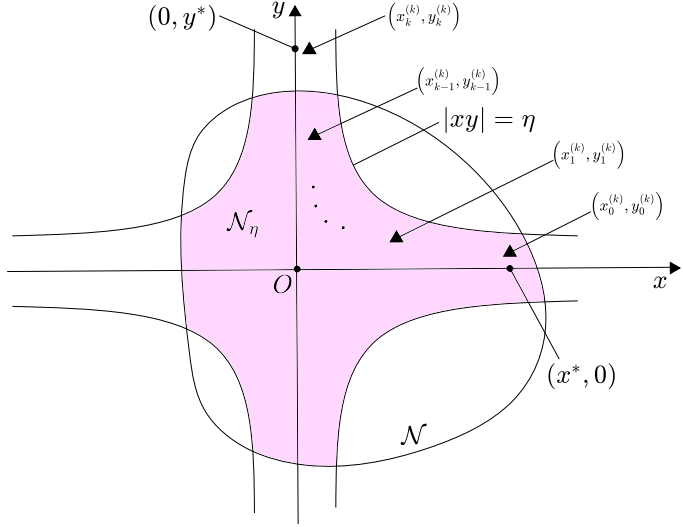}
\caption{Selected points of an {\em ${\rm SR}_k$}-solution (single-round periodic solution satisfying Definition \ref{df:SRkSolution}) in the case $\lambda > 0$ and $\sigma > 0$.
The region $\cN_\eta$ is shaded.}
\label{fig:SRP_defn}
\end{figure}
\begin{definition}
An {\em ${\rm SR}_k$-solution}
is a period-$(k+m)$ solution of $f$
involving $k$ consecutive points in $\cN_\eta$.
\label{df:SRkSolution}
\end{definition}

Below we consider ${\rm SR}_k$-solutions that approach the homoclinic orbit as $k \rightarrow \infty$ (see Theorem \ref{th:necessaryConditions}). We denote points of an ${\rm SR}_k$-solution
by $\left( x^{(k)}_j, y^{(k)}_j \right)$, for $j = 0,\ldots,k+m-1$,
with $\left( x^{(k)}_j, y^{(k)}_j \right) \in \cN_\eta$ for all $j = 0,\ldots,k-1$
as shown in Fig.~\ref{fig:SRP_defn}.
The point $\left(x_{k}^{(k)},y_{k}^{(k)}\right)$ is a fixed point of $T_{0}^{k} \circ T_{1}$, thus the eigenvalues of $D \left( T_0^k \circ T_1 \right) \left( x^{(k)}_k, y^{(k)}_k \right)$
determine the stability of the ${\rm SR}_k$-solution.
Let $\tau_k$ and $\delta_k$ denote the trace and determinant of this matrix, respectively.
If the point $(\tau_k,\delta_k)$ lies in the interior of the triangle shown in Fig.~\ref{fig:stability_triangle}
then the ${\rm SR}_k$-solution is asymptotically stable.
If it lies outside the closure of this triangle then the ${\rm SR}_k$-solution is unstable \cite{El08,LaTr02,Ro04}.
\begin{figure}[h!]
\begin{center}
\includegraphics{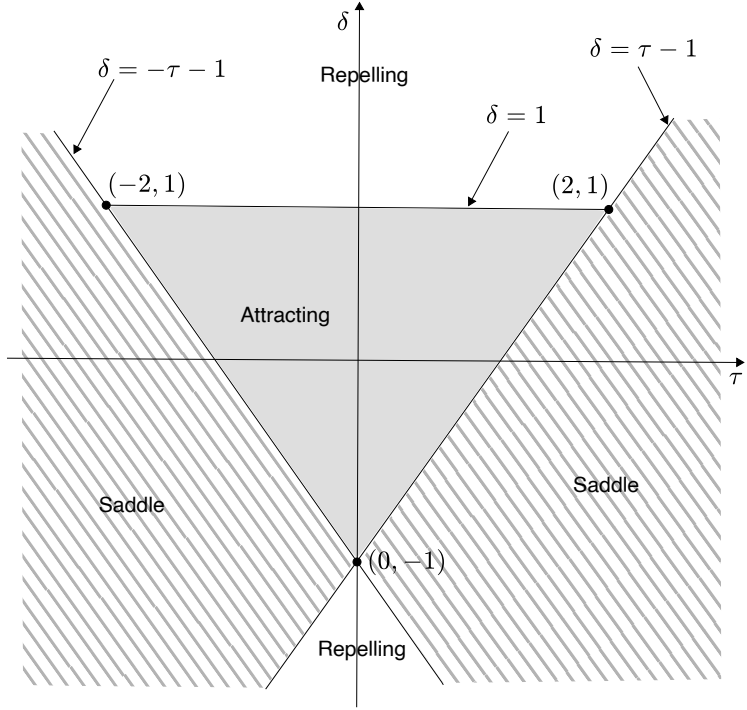}
\caption{The stability of a period-$n$ solution of $f$ in terms of the trace $\tau$ and determinant $\delta$ of the Jacobian matrix $D f^n$ evaluated at one point of the solution.}
\label{fig:stability_triangle}
\end{center}
\end{figure}

\begin{theorem}
Suppose \eqref{eq:N} and \eqref{eq:fourPoints} are satisfied, with $x^*, y^* > 0$,
and $c_1 d_2 - c_2 d_1 \ne 0$. There exists $\eta > 0$ such that
if $f$ has a stable ${\rm SR}_k$-solution for infinitely many values of $k \ge 1$
and points $\left( x^{(k)}_0, y^{(k)}_0 \right)$ of these solutions converge to $(x^*,0)$ as $k \to \infty$, then 
\begin{equation}
d_2 = 0.
\label{eq:tangencyCondition}
\end{equation}
If also $d_5 \ne 0$, then
\begin{equation}
|\lambda \sigma| = 1,
\label{eq:det1Condition}
\end{equation}
and
\begin{equation}
|d_1| = \frac{y^*}{x^*},
\label{eq:globalResonanceCondition}
\end{equation}
with $d_1 = \frac{y^*}{x^*}$ in the case $\lambda \sigma = 1$.
\label{th:necessaryConditions}
\end{theorem}

Theorem \ref{th:necessaryConditions} is proved in $\S\ref{sec:necessaryConditions}$. Here we provide some technical remarks. Equations \eqref{eq:tangencyCondition}--\eqref{eq:globalResonanceCondition}
are independent scalar conditions, thus together represent a codimension-three scenario.
The condition $d_2 = 0$ is equivalent to $(0,y^*)$ being a point of homoclinic tangency.
With also $d_5 \ne 0$ the tangency is quadratic.
The condition $|\lambda \sigma| = 1$ is equivalent to $f$ being area-preversing at the origin.

The condition $|d_1| = \frac{y^*}{x^*}$ is a global condition,
termed {\em global resonance} in the area-preserving setting \cite{GoGo09}. This condition is well-defined because the value of $\frac{d_1 x^*}{y^*}$ is independent of our choice of $x^*$ and $y^*$.
To give geometric meaning to~\eqref{eq:globalResonanceCondition},
consider the perturbed point $(\alpha x^*, y^*)$, where $\alpha \in \mathbb{R}$ is small.
Under $T_1$ this point maps to
\begin{equation}
\Big( x^* \left( 1 + \cO(\alpha) \right), y^* \left( \beta + \cO \left( \alpha^2 \right) \right) \Big),
\nonumber
\end{equation}
where $\beta = \frac{d_1 x^* \alpha}{y^*}$. Notice we are writing the $x$ and $y$-components of these points as multiples of $x^*$ and $y^*$ so that these values provide a relative sense of scale.
Then condition~\eqref{eq:globalResonanceCondition} is equivalent to $|\beta| = |\alpha|$. Therefore~\eqref{eq:globalResonanceCondition} implies that when a point that is perturbed from $(0,y^*)$ in the $x$-direction by an amount $\alpha x^*$
is mapped under $T_1$, the result is a point that is perturbed from $(x^*,0)$ in the $y$-direction by an amount $\pm \alpha y^*$, to leading order.
\subsection{The orientation-preserving case}
\label{sec:sec2p3}

Theorem~\ref{th:necessaryConditions} tells us that the infinite coexistence requires $|\lambda \sigma| = 1$.  Here we consider the case $\lambda \sigma = 1$ for which $f$ is orientation-preserving at the origin.
In this case $\lambda$ and $\sigma$ are resonant and $T_0$ can only be reduced to the form
\begin{equation}
T_0(x,y) = \begin{bmatrix}
\lambda x \left( 1 + x y A(x y) \right) \\
\frac{1}{\lambda} y \left( 1 + x y B(x y) \right)
\end{bmatrix},
\label{eq:T02}
\end{equation}
where $A$ and $B$ are scalar $C^\infty$ functions \cite{GoSh96,St58}.
Let
\begin{align}
a_1 &= A(0), & b_1 &= B(0).
\label{eq:a1b1}
\end{align}
The following theorem tells us that here infinite coexistence is only possible if $a_1 + b_1 = 0$.  This condition is satisfied automatically for area-preserving maps because $\det(D T_0) = 1 + 2 (a_1 + b_1) x y + \cO(x^2 y^2)$. 
\begin{theorem}
Suppose \eqref{eq:N} and \eqref{eq:fourPoints} are satisfied with \eqref{eq:T02} and $x^*, y^* > 0$.  Further suppose $\lambda \sigma = 1$, $d_1 = \frac{y^*}{x^*}$, $d_2 = 0$, and $d_5 \ne 0$.
There exists $\eta > 0$ such that if $f$ 
has a stable ${\rm SR}_k$-solution for infinitely many values of $k \ge 1$
and points $\left( x^{(k)}_0, y^{(k)}_0 \right)$ of these solutions converge to $(x^*,0)$ as $k \to \infty$,
then $a_1 + b_1 = 0$.
\label{th:a1b1}
\end{theorem}
Together Theorems~\ref{th:necessaryConditions} and~\ref{th:a1b1} provide four necessary scalar conditions for infinite coexistence.  The next result provides sufficient conditions and tells us there also exists an infinite family of saddle single-round periodic solutions.
\begin{theorem}
Suppose \eqref{eq:N} and \eqref{eq:fourPoints} are satisfied with \eqref{eq:T02} and $x^*, y^* > 0$. Further suppose $\lambda \sigma = 1$, $d_1 = \frac{y^*}{x^*}$, $d_2 = 0$, $a_1 + b_1 = 0$,
$d_5 \ne 0$, $\Delta > 0$ where
\begin{equation}
\Delta = \left( 1 - \frac{c_2 y^*}{x^*} - \frac{d_4 y^*}{d_1} \right)^2 - 4 d_5 \left( d_3 x^{*^2} + c_1 d_1 x^* \right),
\label{eq:discriminant}
\end{equation}
and
\begin{equation}
-1 < \frac{c_2 y^*}{x^*} < 1 - \frac{\sqrt{\Delta}}{2}.
\label{eq:stabilityCondition}
\end{equation}
Then there exists $k_{\rm min} \in \mathbb{Z}$ such that for all $k \ge k_{\rm min}$
$f$ has an asymptotically stable ${\rm SR}_k$-solution and a saddle ${\rm SR}_k$-solution.
\label{th:sufficientConditionsPositiveCase}
\end{theorem}

From the proof of Theorem~\ref{th:sufficientConditionsPositiveCase} (given below in \S\ref{sec:positiveCase}) it can be seen that the condition $\Delta > 0$ ensures the existence of two ${\rm SR}_k$-solutions for large values of $k$, while~\eqref{eq:stabilityCondition} ensures that one of these is asymptotically stable and the other is a saddle.

\subsection{The orientation-reversing case}
\label{sec:sec2p4}
We now suppose $\lambda \sigma = -1$, so $f$ is orientation-reversing at the origin. In this case $\lambda$ and $\sigma$ are again resonant but $T_0$ can be reduced further than in the $\lambda \sigma = 1$ case.
Specifically we may assume
\begin{equation}
T_0(x,y) = \begin{bmatrix} \lambda x \left( 1 + \cO \left( x^2 y^2 \right) \right) \\
-\tfrac{1}{\lambda} y \left( 1 + \cO \left( x^2 y^2 \right) \right) \end{bmatrix}.
\label{eq:T0rev}
\end{equation}
Analogous to Theorem~\ref{th:sufficientConditionsPositiveCase}, the following result provides sufficient conditions for infinite coexistence.  We see that again it is only required that $\Delta > 0$ and (13) is satisfied. However, here the periodic solutions only exist for either even values of $k$, or odd values of $k$, as determined by the sign of $d_1$.
\begin{theorem}
Suppose \eqref{eq:N} and \eqref{eq:fourPoints} are satisfied with \eqref{eq:T0rev} and $x^*, y^* > 0$. Further suppose $\lambda \sigma = -1$, $|d_1| = \frac{y^*}{x^*}$, $d_2 = 0$,
$d_5 \ne 0$, $\Delta > 0$ (where $\Delta$ is given by \eqref{eq:discriminant})
and \eqref{eq:stabilityCondition} is satisfied.
Then there exists $k_{\rm min} \in \mathbb{Z}$ such that
for all even $k \ge k_{\rm min}$ in the case $d_1 = \frac{y^*}{x^*}$,
and all odd $k \ge k_{\rm min}$ in the case $d_1 = -\frac{y^*}{x^*}$,
$f$ has an asymptotically stable ${\rm SR}_k$-solution
and a saddle ${\rm SR}_k$-solution.
\label{th:sufficientConditionsNegativeCase}
\end{theorem}
\section{Explicit examples of infinite coexistence}
\label{sec:example}

\begin{figure}[b!]
\begin{center}
\includegraphics{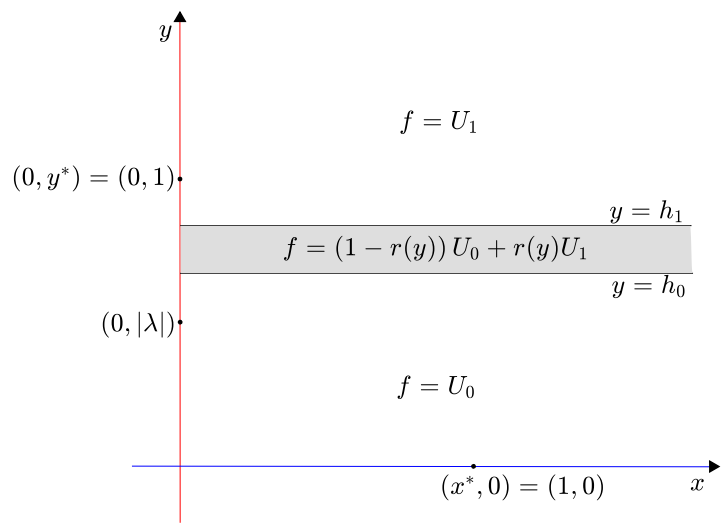}
\caption{The basic structure of the phase space of the map \eqref{eq:fEx}.}
\label{fig:smooth_example}
\end{center}
\end{figure}

Here we demonstrate Theorems \ref{th:sufficientConditionsPositiveCase} and \ref{th:sufficientConditionsNegativeCase}
with piecewise-smooth maps of the form
\begin{equation}
f(x,y) = \begin{cases}
U_0(x,y), & y \le h_0 \,, \\
(1 - r(y)) U_0(x,y) + r(y) U_1(x,y), & h_0 \le y \le h_1 \,, \\
U_1(x,y), & y \ge h_1 \,,
\end{cases}
\label{eq:fEx}
\end{equation}
where
\begin{align}
h_0 &= \frac{2 |\lambda| + 1}{3}, &
h_1 &= \frac{|\lambda| + 2}{3}.
\label{eq:xi0xi1}
\end{align}
This family of maps was motivated by similar families exhibiting infinite coexistence introduced in \cite{GaTr83}. The piecewise nature of \eqref{eq:fEx} is illustrated in Fig.~\ref{fig:smooth_example}.

We let
\begin{equation}
U_0(x,y) = \begin{bmatrix}
\lambda x \\
\sigma y
\end{bmatrix},
\label{eq:U0}
\end{equation}
so that $T_0 = U_0$, and
\begin{equation}
U_1(x,y) = \begin{bmatrix}
1 + c_2 (y - 1) \\
d_1 x + d_5 (y - 1)^2
\end{bmatrix},
\label{eq:U1}
\end{equation}
so that $T_1 = f^m = U_1$ using $m = 1$ and $x^* = y^* = 1$.
Since \eqref{eq:U1} neglects some terms in \eqref{eq:T11},
equation \eqref{eq:discriminant} reduces to $\Delta = (1 - c_2)^2$
and \eqref{eq:stabilityCondition} reduces to $|c_2| < 1$.
Notice that with \eqref{eq:xi0xi1}, equation \eqref{eq:fourPoints} is satisfied
and $(0,y^*) = (0,1)$ lies in the region $y \ge h_1$.

One could choose the function $r$ in \eqref{eq:fEx} such that $f$ is $C^\infty$ on the switching manifolds $y=h_0$ and $y=h_1$, and thus $C^\infty$ on $\mathbb{R}^2$. However, the nature of $f$ outside a neighbourhood of the homoclinic orbit is not important to Theorems \ref{th:sufficientConditionsPositiveCase} and \ref{th:sufficientConditionsNegativeCase} because the ${\rm SR}_k$-solutions converge to the homoclinic orbit, so for simplicity we choose $r$ such that $f$ is $C^1$. This requires
$r(h_0) = 0$,
$r'(h_0) = 0$,
$r(h_1) = 1$, and
$r'(h_1) = 0$.
The unique cubic polynomial satisfying these requirements is
\begin{equation}
r(y) = s \left( \frac{y - h_0}{h_1 - h_0} \right),
\label{eq:r}
\end{equation}
where
\begin{equation}
s(z) = 3 z^2 - 2 z^3,
\label{eq:s}
\end{equation}
see Fig.~\ref{fig:smooth_ry}.
\begin{figure}[h!]
\begin{center}
\includegraphics{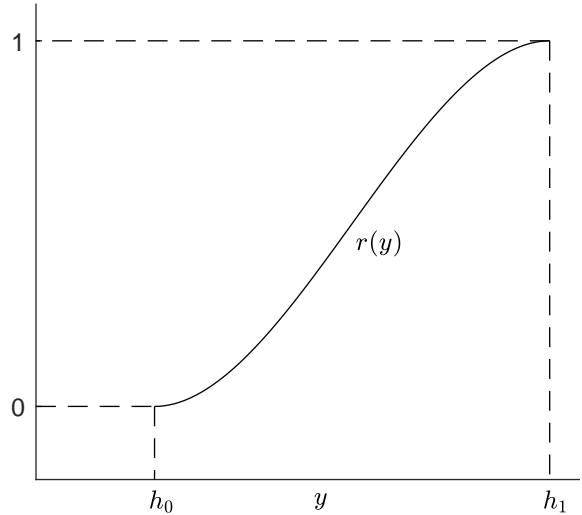}
\caption{The function \eqref{eq:r} (with \eqref{eq:s}) that we use as a convex combination parameter in \eqref{eq:fEx}.}
\label{fig:smooth_ry}
\end{center}
\end{figure}
As examples we fix
\begin{align}
c_2 &= -\tfrac{1}{2}, & d_5 &= 1,
\label{eq:c2d5}
\end{align}
and consider the following four combinations for the remaining parameter values:
\begin{align}
\lambda &= \tfrac{4}{5}, & \sigma &= \tfrac{5}{4}, & d_1 &= 1, \label{eq:param1} \\
\lambda &= -\tfrac{4}{5}, & \sigma &= -\tfrac{5}{4}, & d_1 &= 1, \label{eq:param2} \\
\lambda &= \tfrac{4}{5}, & \sigma &= -\tfrac{5}{4}, & d_1 &= 1, \label{eq:param3} \\
\lambda &= -\tfrac{4}{5}, & \sigma &= \tfrac{5}{4}, & d_1 &= -1. \label{eq:param4}
\end{align}
Fig.~\ref{fig:HT} shows the stable and unstable manifolds of the origin for all four combinations \eqref{eq:param1}--\eqref{eq:param4}. Specifically we have followed the unstable manifold upwards from the origin until observing three
tangential intersections with the stable manifold on the $x$-axis
(some transversal intersections are also visible).
In panels (b) and (c) the unstable manifold evolves outwards along both the positive and negative $y$-axes
because here the unstable eigenvalue $\sigma$ is negative.
In some places the unstable manifolds have extremely high curvature due to the high degree of
nonlinearity of~\eqref{eq:fEx} in the region $h_0 \le y \le h_1$.
Also notice the unstable manifolds have self-intersections because \eqref{eq:fEx} is non-invertible
(for instance $f(0,1) = f(\frac{1}{\lambda},0) = (1,0)$) \cite{CM96}.

With \eqref{eq:param1} and \eqref{eq:param2} the map satisfies the conditions of Theorem \ref{th:sufficientConditionsPositiveCase}, while with \eqref{eq:param3} and \eqref{eq:param4} the map satisfies the conditions of Theorem \ref{th:sufficientConditionsNegativeCase}. Thus in each case \eqref{eq:fEx} has infinitely many asymptotically stable ${\rm SR}_k$-solutions
and these are shown in Fig.~\ref{fig:SRP} up to $k = 15$ (each has period $n = k+1$).
The ${\rm SR}_k$-solutions converge to the homoclinic orbit that includes the points $(0,y^{*})$ and $(x^*,0)$
and occupy only certain quadrants of the $(x,y)$-plane as determined by the signs of $\lambda$ and $\sigma$.
In each case the ${\rm SR}_k$-solutions exist and are stable for relatively low values of $k$
(in fact for all possible $k \ge 0$).
This is due to the simplicity of~\eqref{eq:fEx}; to be clear Theorems \ref{th:sufficientConditionsPositiveCase} and \ref{th:sufficientConditionsNegativeCase} only
tell us about ${\rm SR}_k$-solutions for sufficiently large values of $k$.
In panel (c) the ${\rm SR}_k$-solutions exist for even values of $k$
while in panel (d) they exist for odd values of $k$,
in accordance with Theorem \ref{th:sufficientConditionsNegativeCase}.

Theorems \ref{th:sufficientConditionsPositiveCase} and \ref{th:sufficientConditionsNegativeCase} also guarantee the existence of saddle ${\rm SR}_k$-solutions.
For each of our four cases we show one such solution in Fig.~\ref{fig:SRP}
for the largest value of $k \le 15$. These saddle solutions were straight-forward to compute numerically because
they each involve one point in $y \ge h_1$ and all other points in $y \le h_0$
so from \eqref{eq:U0} and \eqref{eq:U1} their computation reduces to solving a quadratic equation.
For smaller periods some saddle solutions involve points in $h_0 < y < h_1$.
Numerical root-finding methods are needed to compute them and this is beyond the scope of the present paper.
Also panels (b) and (c) each contain an asymptotically stable
double-round periodic solution of period $16$.

Finally in Fig.~\ref{fig:Basin} we show basins of attraction.
These were computed numerically by iterating \eqref{eq:fEx} from a $1000 \times 1000$ grid of initial points.
Each point in the grid is coloured by that of the periodic solution in Fig.~\ref{fig:SRP} to which its forward orbit converges.
If the orbit appeared to converge to some other attractor it is coloured black;
if it appeared to diverge it is coloured white.
In each case the basins are highly intermingled.
We have also observed that many of the boundaries of the basins coincide with the stable manifolds
of the saddle ${\rm SR}_k$-solutions of Theorems \ref{th:sufficientConditionsPositiveCase} and \ref{th:sufficientConditionsNegativeCase}.
\begin{figure}[!htbp]
	\begin{tabular}{c c }

		\hspace{-0.5cm}\includegraphics[width=0.55\textwidth]{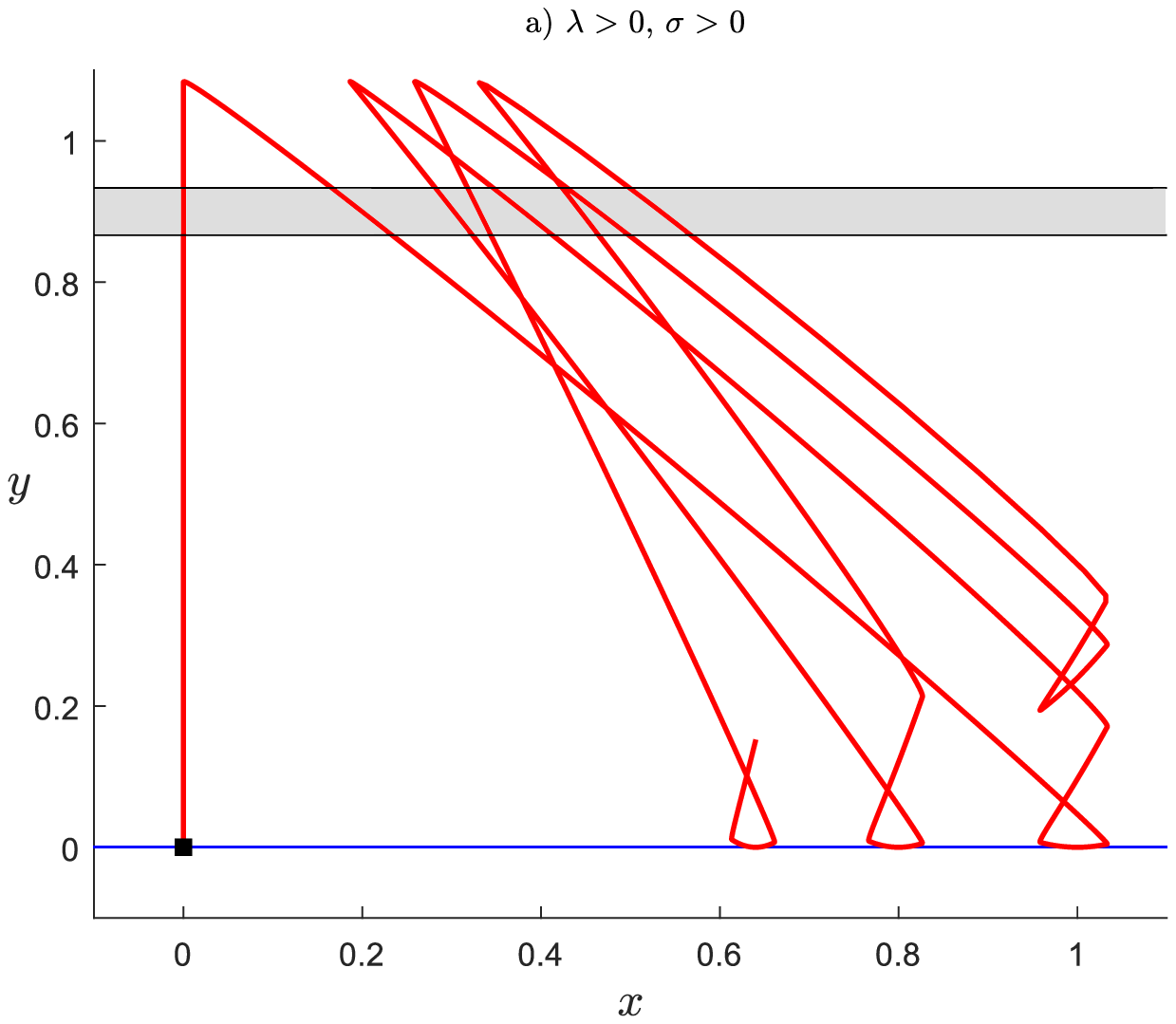}&
		\hspace{-1.3cm}\includegraphics[width=0.55\textwidth]{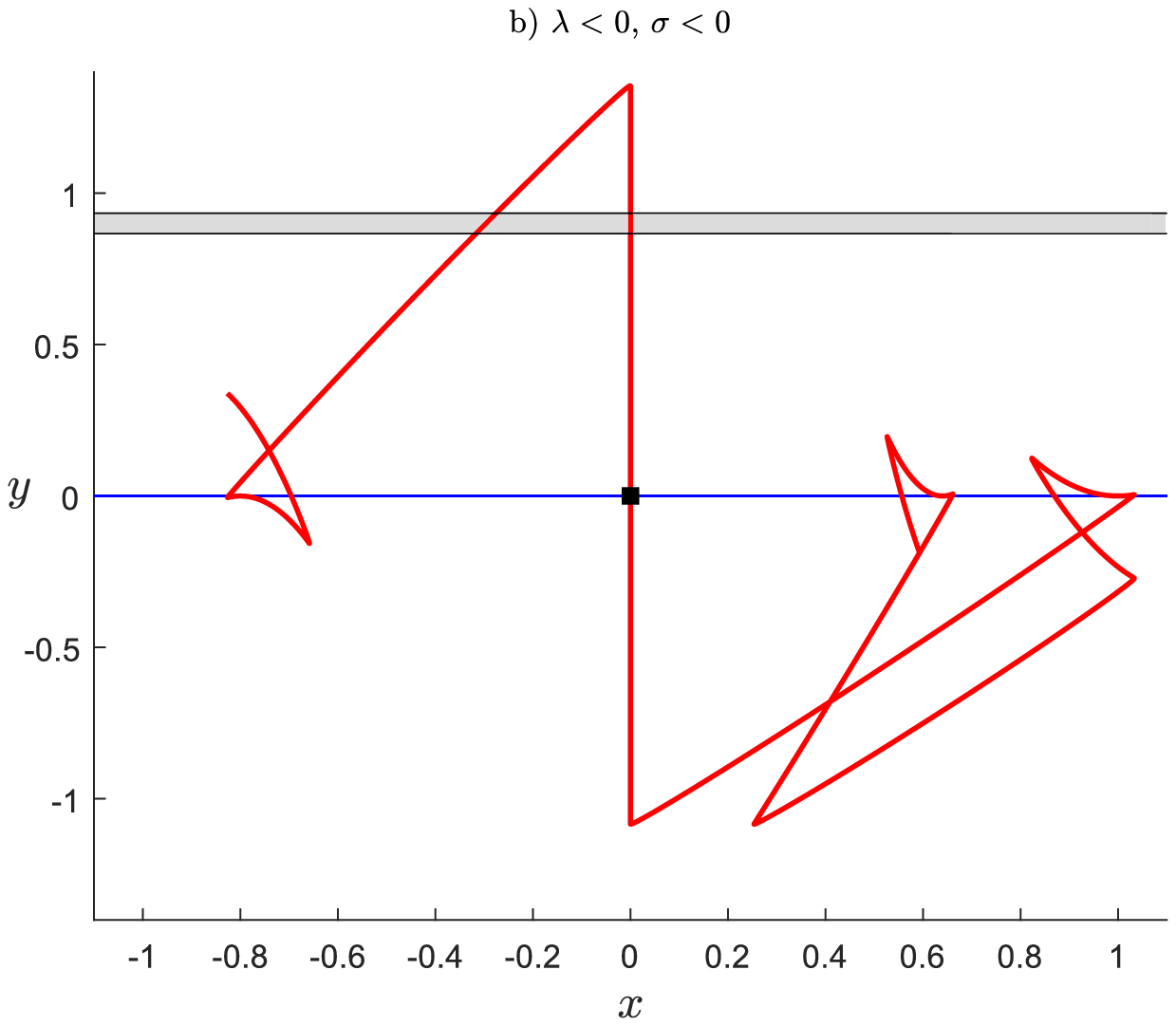}\\
		\hspace{-0.5cm}\includegraphics[width=0.55\textwidth]{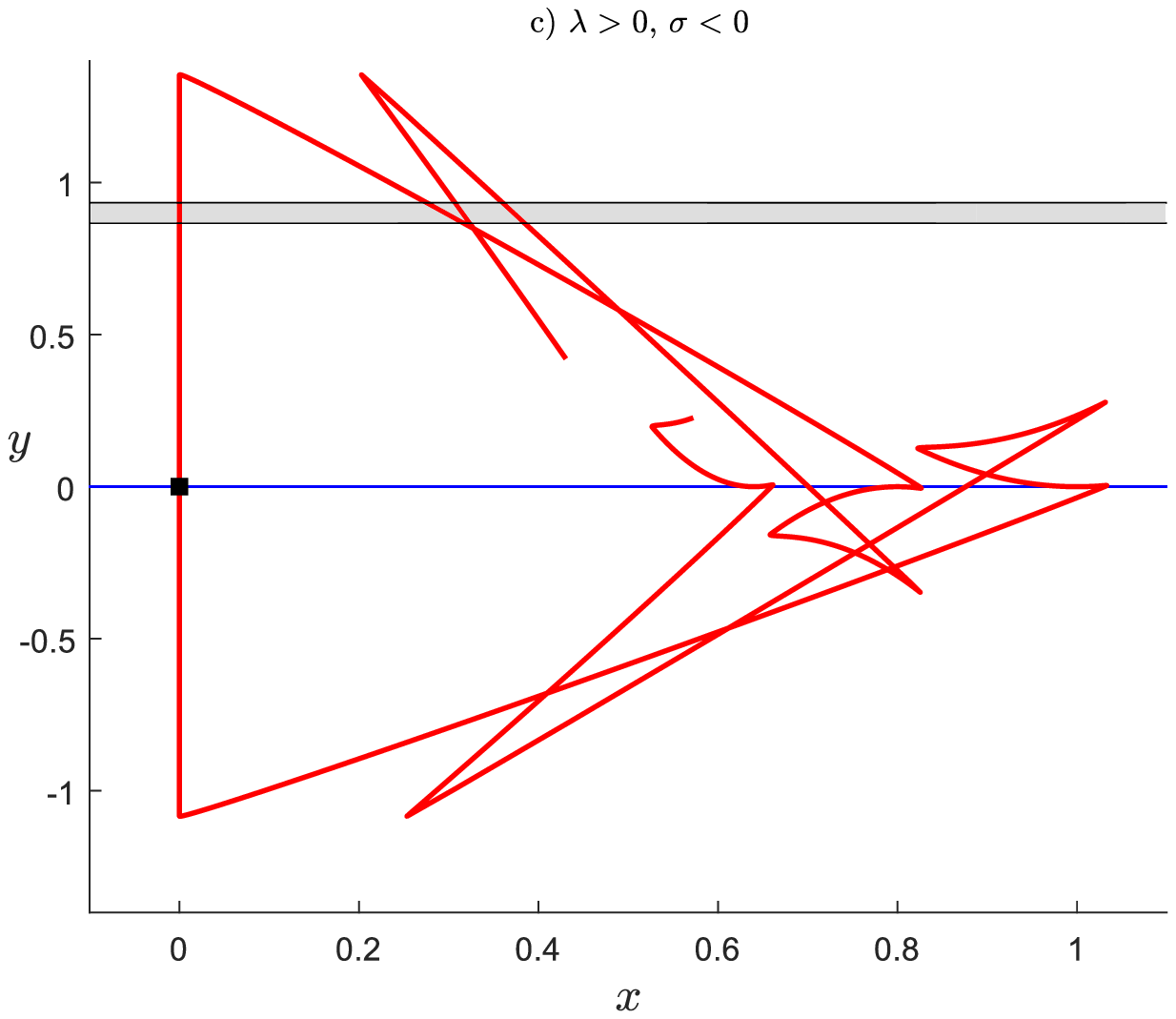}&
		\hspace{-1.3cm}\includegraphics[width=0.55\textwidth]{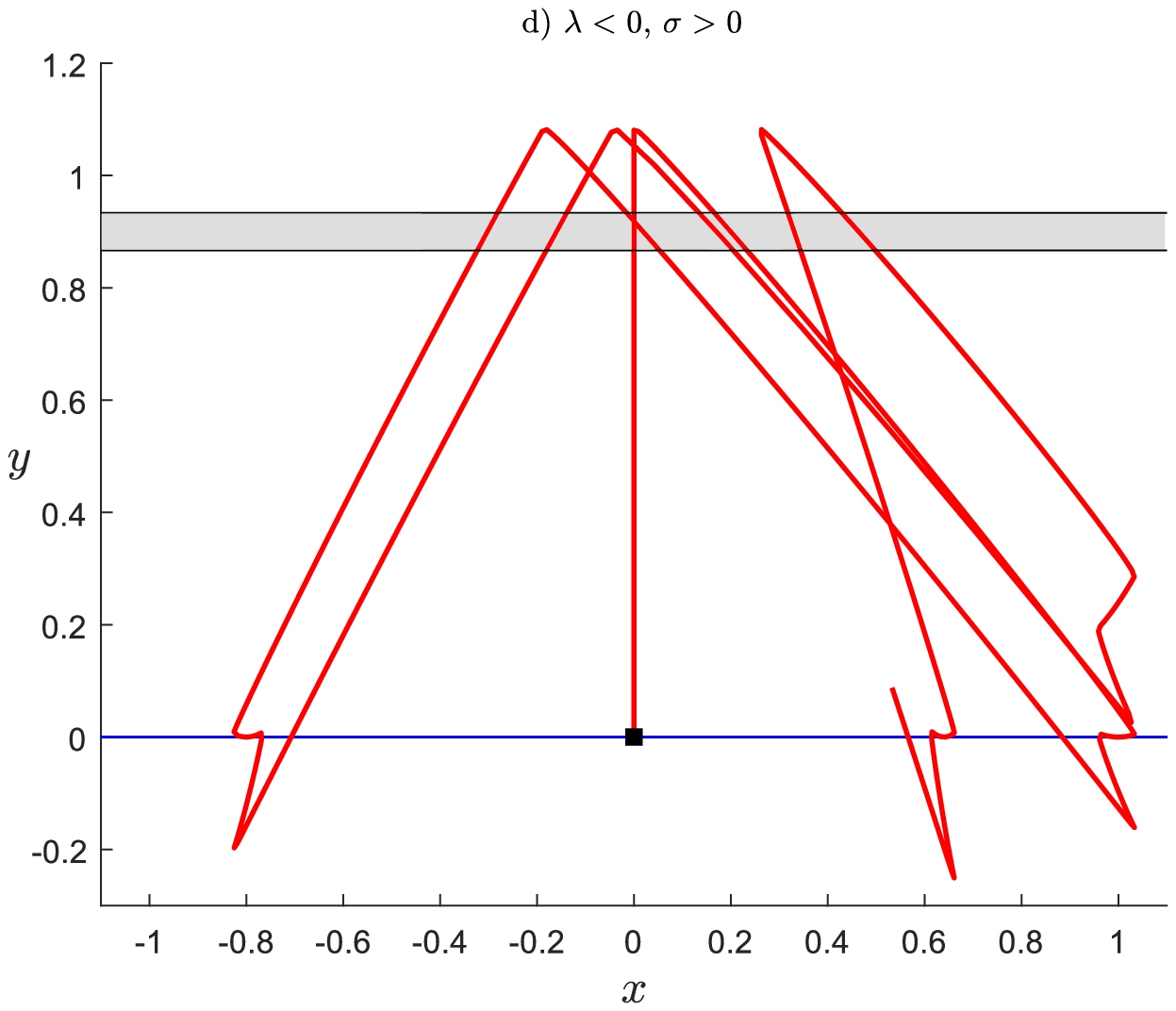}\\
		
	\end{tabular}
	\caption{Parts of the stable [blue] and unstable [red] manifolds of the origin for the map \eqref{eq:fEx} with \eqref{eq:xi0xi1}--\eqref{eq:c2d5}.
Panels (a)--(d) correspond to \eqref{eq:param1}--\eqref{eq:param4} respectively.
In each panel the region $h_0 < y < h_1$ is shaded.
}
	\label{fig:HT}
\end{figure}

\begin{figure}[!htbp]
	\begin{tabular}{c c }

		\hspace{-0.3cm}\includegraphics[width=0.5\textwidth]{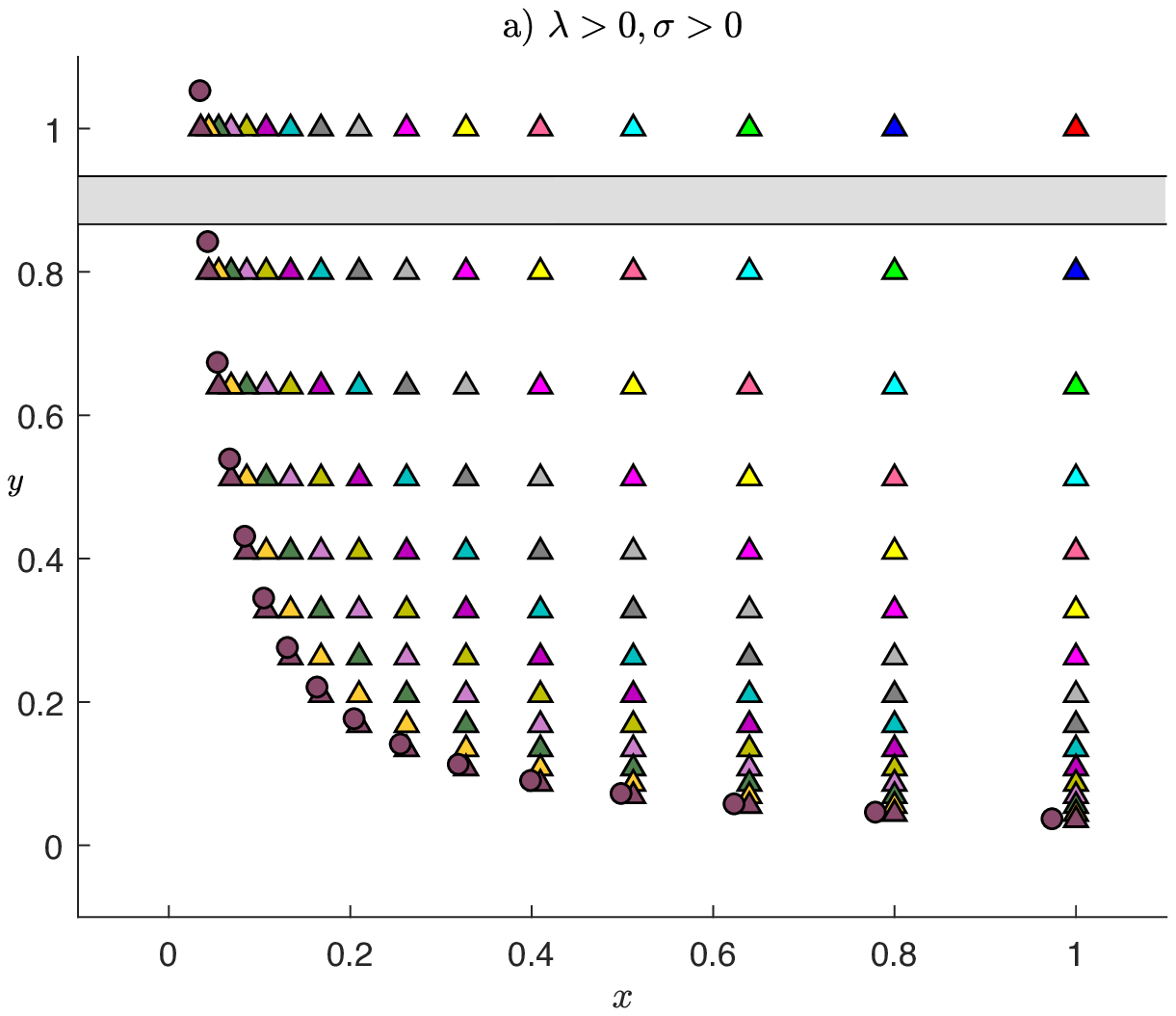}&
		
		\hspace{-0.1cm}\includegraphics[width=0.5\textwidth]{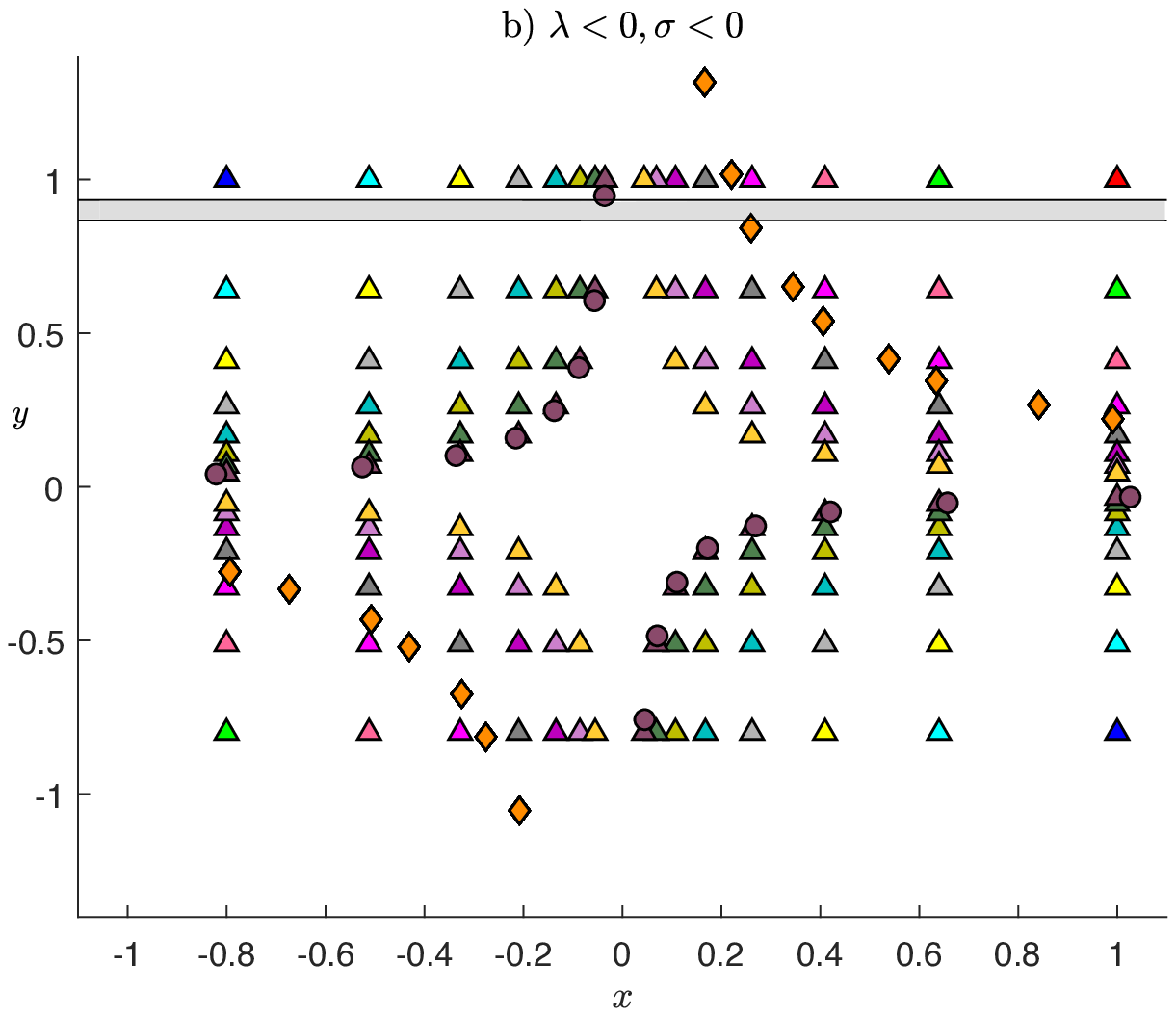}\\

		\hspace{-0.3cm}\includegraphics[width=0.5\textwidth]{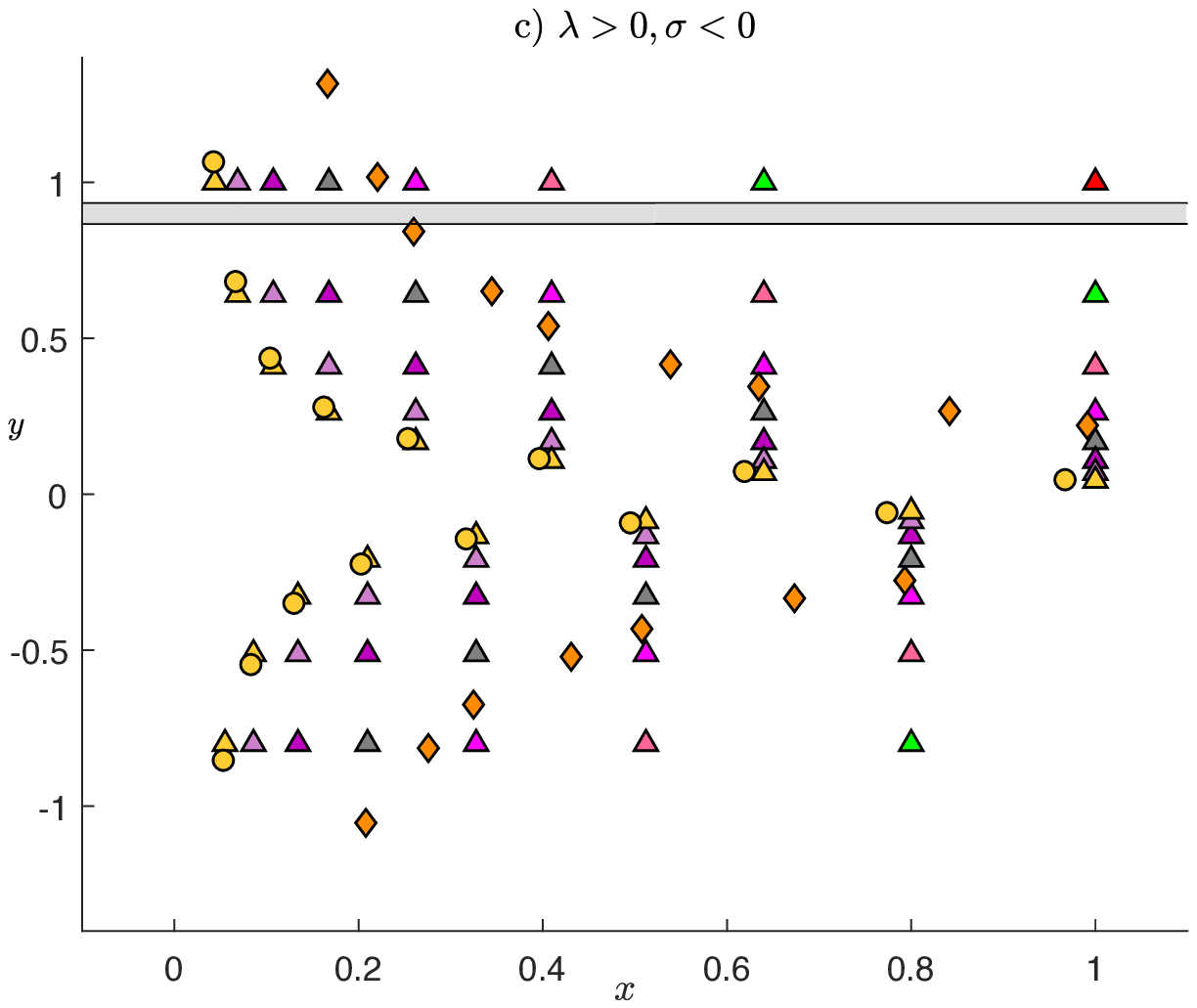}&
	
		\hspace{-0.1cm}\includegraphics[width=0.5\textwidth]{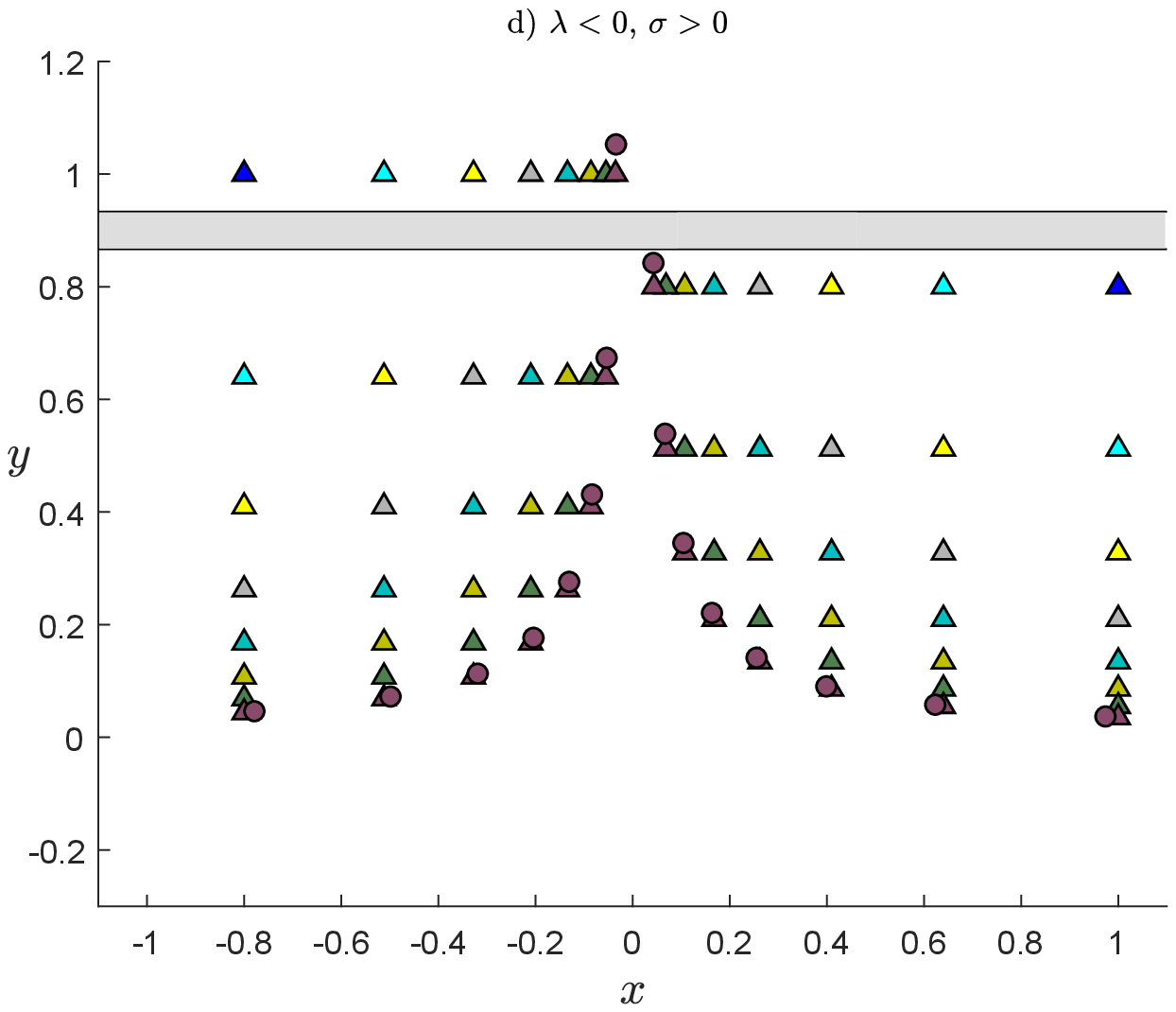}\\
		
	\end{tabular}
	\hspace*{5.5cm}\includegraphics{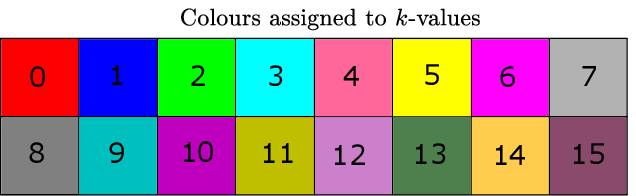}
	\vspace{0.3cm}\caption{Asymptotically stable ${\rm SR}_k$-solutions of \eqref{eq:fEx} with \eqref{eq:xi0xi1}--\eqref{eq:c2d5}.
Panels (a)--(d) correspond to \eqref{eq:param1}--\eqref{eq:param4} respectively.
Points of the stable ${\rm SR}_k$-solutions are indicated by triangles
and coloured by the value of $k$ (as indicated in the key).
In panels (a) and (b) the solutions are shown for $k=0$ (a fixed point in $y > h_1$) up to $k = 15$.
In panel (c) the solutions are shown for $k = 0,2,4,\ldots,14$
and in panel (d) the solutions are shown for $k = 1,3,5,\ldots,15$.
In each panel one saddle ${\rm SR}_k$-solution is shown with circles
(with $k = 14$ in panel (c) and $k = 15$ in the other panels).
In panels (b) and (c) asymptotically stable double-round periodic solutions are shown with diamonds.
} 
	\label{fig:SRP}
\end{figure}
\begin{figure}[!htbp]
	\begin{tabular}{c c}
					
		\hspace{-0.5cm}\includegraphics[width=0.55\textwidth]{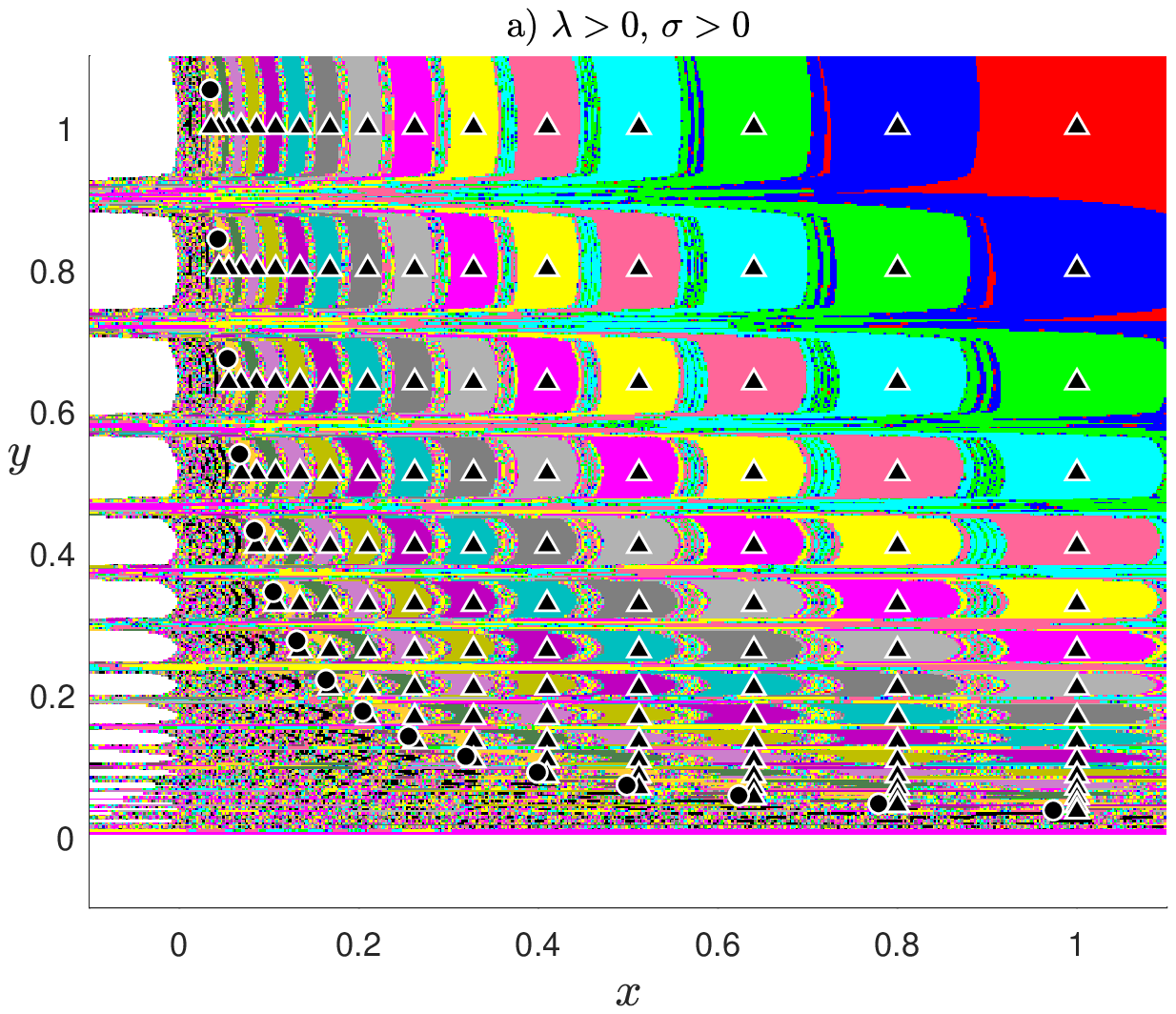}&
		
		\hspace{-1.3cm}\includegraphics[width=0.55\textwidth]{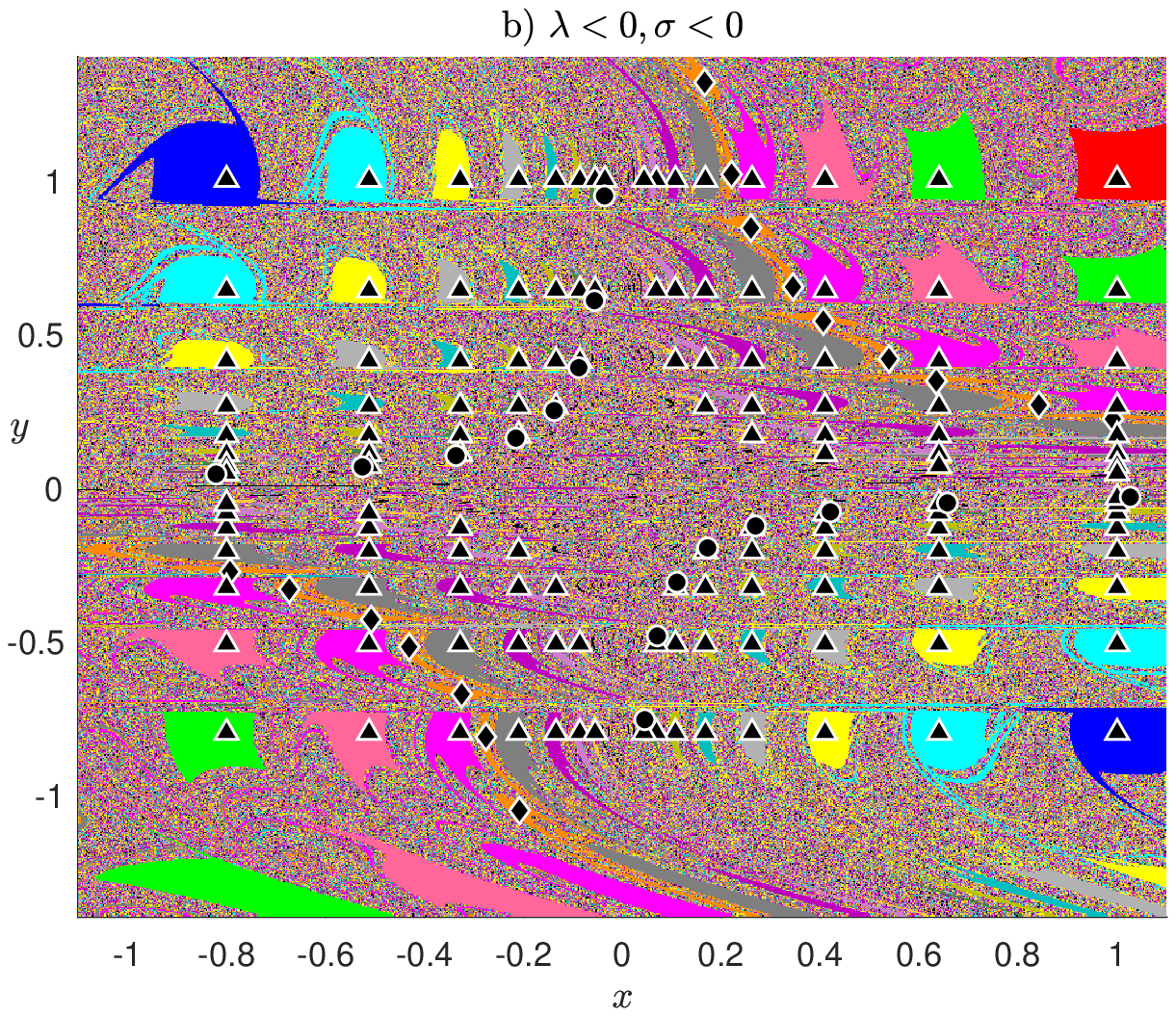}\\
	
		\hspace{-0.5cm}\includegraphics[width=0.55\textwidth]{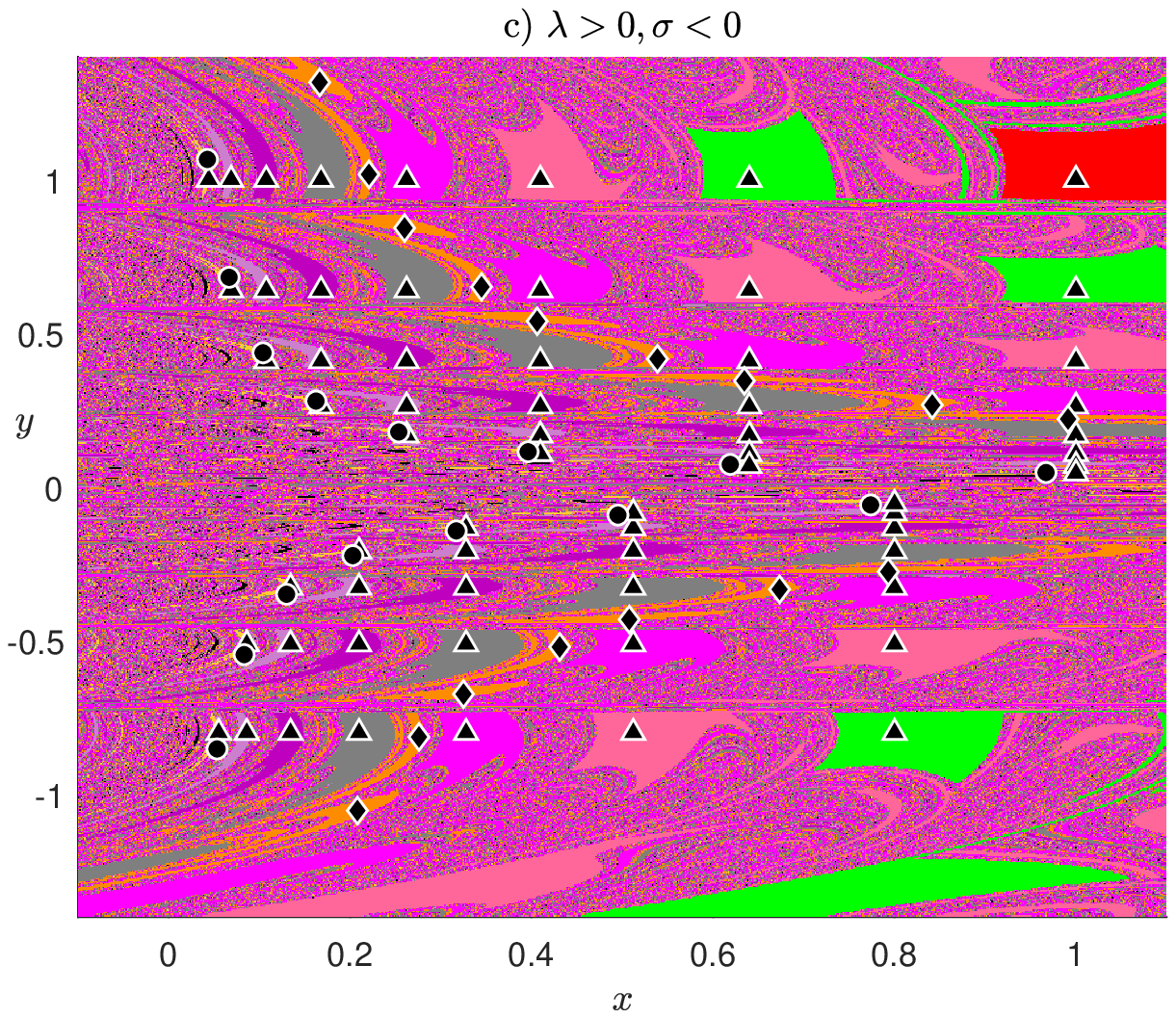}&
		\hspace{-1.3cm}\includegraphics[width=0.55\textwidth]{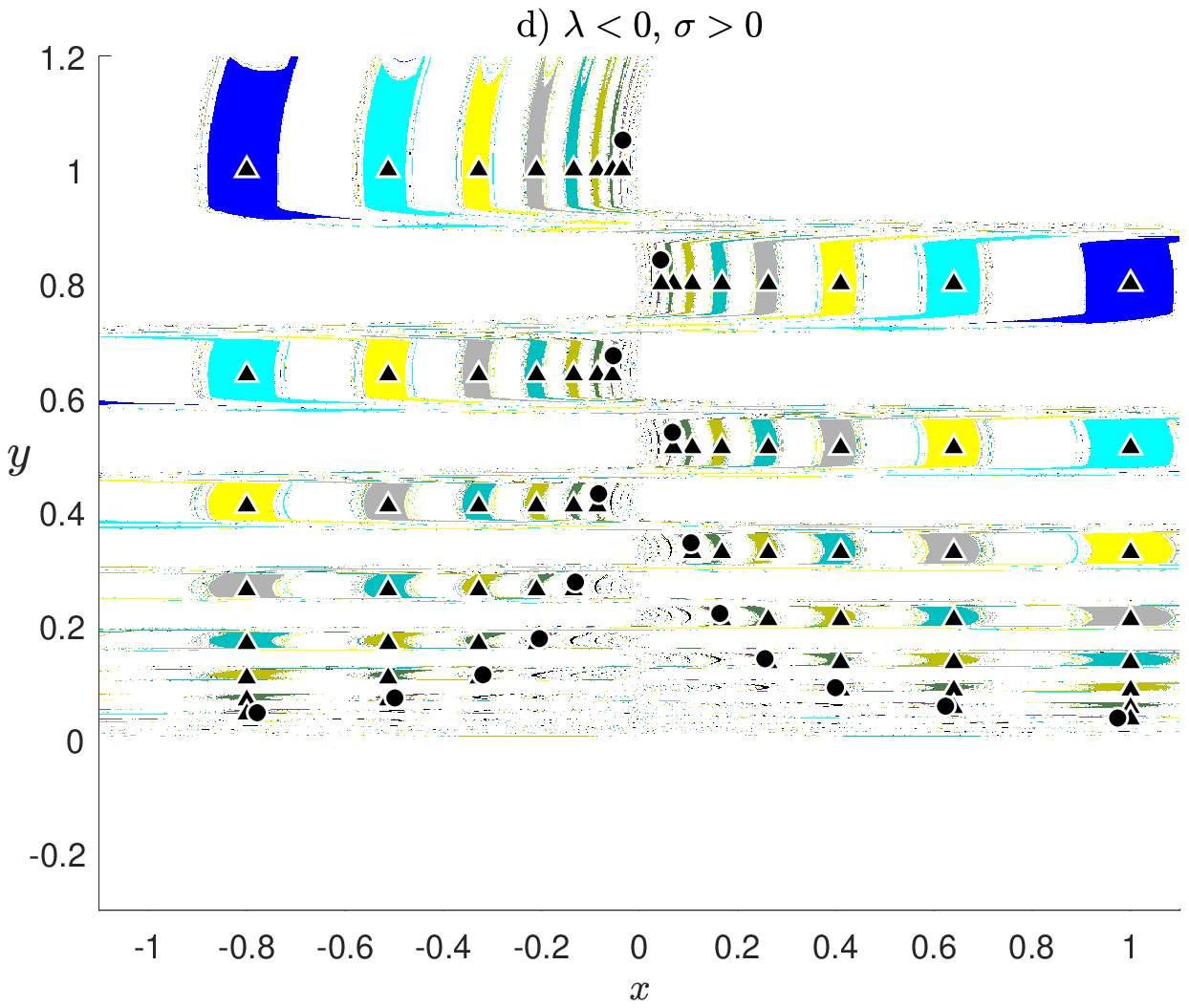}\\
		
	\end{tabular}
    \hspace*{5.5cm}\includegraphics{colour-legend-periodic.eps}
	\vspace*{0.3cm}\caption{Basins of attraction for the asymptotically stable ${\rm SR}_k$-solutions shown in Fig \ref{fig:SRP}. Specifically each point in a $1000 \times 1000$ grid is coloured by that of the ${\rm SR}_k$-solution to which its forward orbit under $f$ converges to.} 
	\label{fig:Basin}
\end{figure}

\section{Derivation of necessary conditions for infinite coexistence}
\label{sec:necessaryConditions}
In this section we work towards a proof of Theorem \ref{th:necessaryConditions}.  Since an ${\rm SR}_k$-solution has period $k + m$, its stability is determined by the eigenvalues of $D f^{k+m}$ evaluated at any point of the solution.  Below we work with the point $\left( x^{(k)}_k, y^{(k)}_k \right)$.  Since $\left( x^{(k)}_k, y^{(k)}_k \right) = T_0^k \left( x^{(k)}_0, y^{(k)}_0 \right)$, and $\left( x^{(k)}_0, y^{(k)}_0 \right) = T_1 \left( x^{(k)}_k, y^{(k)}_k \right)$ (see Fig.~\ref{fig:SRP_defn}), we have
\begin{equation}
D f^{k+m} \left( x^{(k)}_k, y^{(k)}_k \right) = D T_0^k \left( x^{(k)}_0, y^{(k)}_0 \right)
D T_1 \left( x^{(k)}_k, y^{(k)}_k \right).
\label{eq:DfkPlusm}
\end{equation}

To obtain information on the eigenvalues of \eqref{eq:DfkPlusm} we first construct bounds on the values of the points of an ${\rm SR}_k$-solution (Lemmas \ref{le:y0Bound} and \ref{le:xjyjBound}).  We then estimate the entries of $D T_0^k \left( x^{(k)}_0, y^{(k)}_0 \right)$ (Lemmas \ref{le:MjBoundsl} and \ref{le:PjBoundsl}).  Next we estimate the contribution of the resonant terms in $T_0^k$ (Lemma \ref{le:T0k}).  These finally enable us to prove Theorem \ref{th:necessaryConditions}.  Essentially we show that if the conditions \eqref{eq:tangencyCondition}--\eqref{eq:globalResonanceCondition} are not all met then the trace of \eqref{eq:DfkPlusm}, denoted $\tau_k$ above, diverges as $k \to \infty$.  This implies that the eigenvalues of  \eqref{eq:DfkPlusm} cannot lie inside the shaded region of Fig.~\ref{fig:stability_triangle} for more than finitely many values of $k$.


We first observe that the resonant terms $F$ and $G$ of \eqref{eq:T01} are continuous and $\cN$ is bounded, so there exists $R > 0$ such that
\begin{equation}
|F(x,y)| \le R, \quad |G(x,y)| \le R, \quad \text{for all}~ (x,y) \in \cN.
\label{eq:FGBound}
\end{equation}
Let
\begin{equation}
\eta_R = \frac{1 - |\sigma|^{-\frac{1}{2}}}{R},
\label{eq:xyBound}
\end{equation}
and notice $\eta_R > 0$.

\begin{lemma}
Suppose an infinite family of ${\rm SR}_k$-solutions with $\eta = \eta_R$
satisfies $\left( x^{(k)}_0, y^{(k)}_0 \right) \to (x^*,0)$ as $k \to \infty$.
Then
\begin{equation}
\left| y^{(k)}_0 \right| \le 2 y^* |\sigma|^{-\frac{k}{2}},
\label{eq:y0Bound}
\end{equation}
for all sufficiently large values of $k$.
\label{le:y0Bound}
\end{lemma}

\begin{proof}
By assumption $\left( x^{(k)}_j, y^{(k)}_j \right) \in \cN_{\eta_R}$
for each $j \in \{ 0,\ldots,k-1 \}$,
thus from \eqref{eq:T01}, \eqref{eq:FGBound}, and \eqref{eq:xyBound} we obtain
\begin{align*}
\left| y^{(k)}_{j+1} \right| &= |\sigma| \left| y^{(k)}_j \right|
\left| 1 + x^{(k)}_j y^{(k)}_j G \left( x^{(k)}_j, y^{(k)}_j \right) \right| \\
&\ge |\sigma| \left| y^{(k)}_j \right| \left( 1 - \eta_R R \right) \\
&= |\sigma|^{\frac{1}{2}} \left| y^{(k)}_j \right|.
\end{align*}
By applying this bound $k$ times we obtain
\begin{equation}
\left| y^{(k)}_0 \right| \le |\sigma|^{-\frac{k}{2}} \left| y^{(k)}_k \right|.
\nonumber
\end{equation}
But $y^{(k)}_k \to y^*$ as $k \to \infty$ thus
$\left| y^{(k)}_k \right| \le 2 y^*$, say, for all sufficiently large values of $k$.
This verifies \eqref{eq:y0Bound}.
\end{proof}

\begin{lemma}
Suppose $|\lambda \sigma| \le 1$ and suppose an infinite family of ${\rm SR}_k$-solutions with $\eta = \eta_R$
satisfies $\left( x^{(k)}_0, y^{(k)}_0 \right) \to (x^*,0)$ as $k \to \infty$.
Then there exists $\omega > 0$ such that
\begin{equation}
\begin{split}
\left| x^{(k)}_j \right| &\le \omega |\lambda|^j, \\
\left| y^{(k)}_j \right| &\le \omega |\sigma|^{j-k},
\label{eq:xjyjBound}
\end{split}
\end{equation}
for all $j \in \{ 0, \ldots, k \}$ and all sufficiently large values of $k$.
\label{le:xjyjBound}
\end{lemma}

\begin{proof}
By Lemma \ref{le:y0Bound} there exists $\omega_0 > 0$ such that
\begin{align}
\left| x^{(k)}_0 \right| &\le \omega_0 \,,
\label{eq:x0Bound2} \\
\left| y^{(k)}_0 \right| &\le \omega_0 |\sigma|^{-\frac{k}{2}},
\label{eq:y0Bound2}
\end{align}
for all sufficiently large values of $k$.
For the remainder of the proof we assume $k$ is sufficiently large that
\begin{align}
\left( 1 + 4 \omega_0^2 R |\sigma|^{-\frac{k}{2}} \right)^k &< 2,
\label{eq:powerkBoundUpper} \\
\left( 1 - 4 \omega_0^2 R |\sigma|^{-\frac{k}{2}} \right)^k &> \tfrac{1}{2}.
\label{eq:powerkBoundLower}
\end{align}
Below we use induction on $j$ to prove that
\begin{align}
\left| x^{(k)}_0 \right| |\lambda|^j \left( 1 - 4 \omega_0^2 R |\sigma|^{-\frac{k}{2}} \right)^j
\le \left| x^{(k)}_j \right| \le
\left| x^{(k)}_0 \right| |\lambda|^j \left( 1 + 4 \omega_0^2 R |\sigma|^{-\frac{k}{2}} \right)^j,
\label{eq:xjDoubleBound} \\
\left| y^{(k)}_0 \right| |\sigma|^j \left( 1 - 4 \omega_0^2 R |\sigma|^{-\frac{k}{2}} \right)^j
\le \left| y^{(k)}_j \right| \le
\left| y^{(k)}_0 \right| |\sigma|^j \left( 1 + 4 \omega_0^2 R |\sigma|^{-\frac{k}{2}} \right)^j,
\label{eq:yjDoubleBound}
\end{align}
for all $j \in \{ 0,\ldots,k \}$.
This will complete the proof because, first, \eqref{eq:x0Bound2},
\eqref{eq:powerkBoundUpper}, and \eqref{eq:xjDoubleBound} combine to produce
\begin{equation}
\left| x^{(k)}_j \right| \le 2 \omega_0 |\lambda|^j.
\nonumber
\end{equation}
Second \eqref{eq:powerkBoundLower} and \eqref{eq:yjDoubleBound} evaluated at $j=k$ combine to produce
$\frac{1}{2} \left| y^{(k)}_0 \right| |\sigma|^k \le \left| y^{(k)}_k \right|$.
But $y^{(k)}_k \to y^*$ as $k \to \infty$, thus
$\left| y^{(k)}_k \right| \le 2 y^*$ for sufficiently large values of $k$, and so
$\left| y^{(k)}_0 \right| \le 4 y^* |\sigma|^{-k}$.
Hence by \eqref{eq:powerkBoundUpper} and \eqref{eq:yjDoubleBound} we have
\begin{equation}
\left| y^{(k)}_j \right| \le 8 y^* |\sigma|^{j-k}.
\nonumber
\end{equation}

It remains to verify \eqref{eq:xjDoubleBound}--\eqref{eq:yjDoubleBound}.
Trivially these hold for $j=0$.
Suppose \eqref{eq:xjDoubleBound}--\eqref{eq:yjDoubleBound} hold for some $j \in \{ 0,\ldots,k-1 \}$
(this is our induction hypothesis).
Then by \eqref{eq:T01},
\begin{align}
|\lambda| \left| x^{(k)}_j \right| \left( 1 - \left| x^{(k)}_j y^{(k)}_j \right| R \right)
&\le \left| x^{(k)}_{j+1} \right| \le
|\lambda| \left| x^{(k)}_j \right| \left( 1 + \left| x^{(k)}_j y^{(k)}_j \right| R \right),
\label{eq:xjDoubleBoundProof} \\
|\sigma| \left| y^{(k)}_j \right| \left( 1 - \left| x^{(k)}_j y^{(k)}_j \right| R \right)
&\le \left| y^{(k)}_{j+1} \right| \le
|\sigma| \left| y^{(k)}_j \right| \left( 1 + \left| x^{(k)}_j y^{(k)}_j \right| R \right).
\label{eq:yjDoubleBoundProof}
\end{align}
The induction hypothesis implies
\begin{align}
\left| x^{(k)}_j y^{(k)}_j \right|
&\le \left| x^{(k)}_0 \right| \left| y^{(k)}_0 \right| |\lambda \sigma|^j
\left( 1 + 4 \omega_0^2 R |\sigma|^{-\frac{k}{2}} \right)^{2 j} \nonumber \\
&\le 4 \omega_0^2 |\sigma|^{-\frac{k}{2}},
\label{eq:xjyjBoundProof}
\end{align}
where we have used \eqref{eq:x0Bound2}, \eqref{eq:y0Bound2}, \eqref{eq:powerkBoundUpper}
and $|\lambda \sigma| \le 1$ in the second line.
By applying \eqref{eq:xjyjBoundProof} and
the induction hypothesis to \eqref{eq:xjDoubleBoundProof}--\eqref{eq:yjDoubleBoundProof}
we obtain \eqref{eq:xjDoubleBound}--\eqref{eq:yjDoubleBound} for $j+1$,
and this completes the induction step.
\end{proof}
For brevity we write
\begin{equation}
M_j = D T_0 \left( x^{(k)}_j, y^{(k)}_j \right)
= \begin{bmatrix} p_j & q_j \\ r_j & s_j \end{bmatrix},
\label{eq:Mjl}
\end{equation}
and
\begin{equation}
P_j = M_{j-1} \cdots M_1 M_0
= \begin{bmatrix} t_j & u_j \\ v_j & w_j \end{bmatrix}.
\label{eq:Pjl}
\end{equation}
\begin{lemma}
Suppose $|\lambda \sigma| \le 1$ and suppose an infinite family of ${\rm SR}_k$-solutions with $\eta = \eta_R$
satisfies $\left( x^{(k)}_0, y^{(k)}_0 \right) \to (x^*,0)$ as $k \to \infty$.
Then there exists $\alpha > 0$ such that
\begin{equation}
\begin{aligned}
|p_j - \lambda| &\le \alpha |\sigma|^{-k}, &
|q_j| &\le \alpha |\sigma|^{-2 j}, \\
|r_j| &\le \alpha |\sigma|^{2 (j-k)}, &
\hspace{10mm}
\left|s_j - \sigma \right| &\le \alpha |\sigma|^{-k},
\end{aligned}
\label{eq:MjBoundsl}
\end{equation}
for all $j \in \{ 0,\ldots,k-1 \}$ and all sufficiently large values of $k$.
\label{le:MjBoundsl}
\end{lemma}

\begin{proof}
By Lemma \ref{le:xjyjBound} and $|\lambda \sigma| \leq 1$ there exists $\omega > 0$ such that
\begin{align}
\left| x^{(k)}_j \right| &\le \omega |\sigma|^{-j}, &
\left| y^{(k)}_j \right| &\le \omega |\sigma|^{j-k},
\label{eq:xjyjBound2l}
\end{align}
for all sufficiently large values of $k$.
By differentiating \eqref{eq:T01} we obtain
\begin{equation}
D T_0(x,y) = \begin{bmatrix}
\lambda \big( 1 + x y (2 F + x \frac{\partial{F}}{\partial{x}}) \big) & \lambda x^2 (F +  y \frac{\partial{F}}{\partial{y}})\\
\sigma y^2 (G + x \frac{\partial{G}}{\partial{x}}) & \sigma \big( 1 + x y (2 G + y \frac{\partial{G}}{\partial{y}}) \big)
\end{bmatrix}.
\label{eq:DT0l}
\end{equation}
Since $F$, $G$, and their derivatives are continuous and $\cN_{\eta_R}$ is bounded
there exists $\alpha > 0$ such that throughout $\cN_{\eta_R}$ we have that
$\alpha$ is greater than
$\omega^2 |\lambda| |2 F + x \frac{\partial{F}}{\partial{x}}|$,
$\omega^2 |\lambda| |F +  y \frac{\partial{F}}{\partial{y}}|$,
$|\sigma| \omega^2 |G + x \frac{\partial{G}}{\partial{x}}|$, and
$|\sigma| \omega^2 |2 G +  y \frac{\partial{G}}{\partial{y}}|$.
Then \eqref{eq:MjBoundsl} follows immediately from
\eqref{eq:xjyjBound2l} and \eqref{eq:DT0l}.
\end{proof}

\begin{lemma}
Suppose $|\lambda \sigma| \le 1$ and suppose an infinite family of ${\rm SR}_k$-solutions with $\eta = \eta_R$
satisfies $\left( x^{(k)}_0, y^{(k)}_0 \right) \to (x^*,0)$ as $k \to \infty$.
Then there exists $\beta > 0$ such that
\begin{equation}
\begin{aligned}
|t_j - \lambda^j| &\le \beta j |\sigma|^{-(j+k)}, &
|u_j| &\le \beta j |\sigma|^{-j}, \\
|v_j| &\le \beta j |\sigma|^{j - 2 k}, &
\hspace{10mm}
\left|w_j - \sigma^{j} \right| &\le \beta j |\sigma|^{j-k},
\end{aligned}
\label{eq:PjBoundsl}
\end{equation}
for all $j \in \{ 1,\ldots,k \}$ and all sufficiently large values of $k$.
\label{le:PjBoundsl}
\end{lemma}

\begin{proof}
By Lemma \ref{le:xjyjBound} and $|\lambda\sigma| \leq 1$ there exists $\omega > 0$ satisfying
\eqref{eq:xjyjBound2l} for all sufficiently large values of $k$.
Let $\alpha > 0$ satisfy \eqref{eq:MjBoundsl}
and let $\beta = \frac{2 \alpha}{|\lambda|}$.

We now verify \eqref{eq:PjBoundsl} by induction on $j$.
Observe $P_1 = M_0$ so \eqref{eq:PjBoundsl} holds with $j = 1$
because \eqref{eq:MjBoundsl} holds with $j = 0$ and $\beta > \alpha$.
Now suppose \eqref{eq:PjBoundsl} holds for some $j \in \{ 1,\ldots,k-1 \}$
(this is our induction hypothesis).
Since $P_{j+1} = M_j P_j$ we have
\begin{equation}
\begin{bmatrix} t_{j+1} & u_{j+1} \\ v_{j+1} & w_{j+1} \end{bmatrix} =
\begin{bmatrix}
p_j t_j + q_j v_j & p_j u_j + q_j w_j \\
r_j t_j + s_j v_j & r_j u_j + s_j w_j
\end{bmatrix}.
\label{eq:Pjp1l}
\end{equation}
We now verify the four inequalities \eqref{eq:PjBoundsl} for $j+1$ in order.
First
\begin{align}
\left| t_{j+1} - \lambda^{j+1} \right|
&= \left| p_j t_j + q_j v_j - \lambda^{j+1} \right| \nonumber \\
&= \left| (p_j - \lambda) \left( t_j - \lambda^j \right) + (p_j - \lambda) \lambda^j
+ \lambda \left( t_j - \lambda^j \right) + q_j v_j \right| \nonumber \\
&\le |p_j - \lambda| \left| t_j - \lambda^j \right| + |p_j - \lambda| |\lambda|^j
+ |\lambda| \left| t_j - \lambda^j \right| + |q_j| |v_j|, \nonumber
\end{align}
and so by \eqref{eq:MjBoundsl} and the induction hypothesis we have
\begin{equation}
\left| t_{j+1} - \lambda^{j+1} \right|
\le \beta j |\sigma|^{-(k+j+1)}
+ \left( 2 \alpha \beta j |\sigma|^{-k+1} + \alpha|\sigma| \right) |\sigma|^{-(k+j+1)}.
\nonumber
\end{equation}
Thus for sufficiently large values of $k$,
\begin{align}
\left| t_{j+1} - \lambda^{j+1} \right|
&\le \beta j |\sigma|^{-(k+j+1)}
+ \tfrac{2 \alpha}{|\lambda|} |\sigma|^{-(k+j+1)}
\nonumber \\
&= \beta (j+1) |\sigma|^{-(k+j+1)},
\nonumber
\end{align}
where we have substituted our formula for $\beta$ in the last line.
In a similar fashion we obtain
\begin{align}
\left| u_{j+1} \right|
&= \left| p_j u_j + q_j w_j \right| \nonumber \\
&\le |p_j| |u_j| + |q_j| |w_j| \nonumber \\
&\le \left( |\lambda| + \alpha |\sigma|^{-k} \right) \beta j |\sigma|^{-j}
+ \alpha |\sigma|^{-2 j} \left( |\sigma|^{j} + \beta j |\sigma|^{j-k} \right) \nonumber \\
&= \beta j |\sigma|^{-(j+1)} + \left(2\alpha\beta j  |\sigma|^{-k+1} + \alpha |\sigma|\right)
|\sigma|^{-(j+1)}\\
&\le \beta j |\sigma|^{-(j+1)} + \tfrac{2 \alpha}{|\lambda|} |\sigma|^{-(j+1)} \nonumber \\
&= \beta (j+1) |\sigma|^{-(j+1)}, \nonumber
\end{align}
\begin{align}
\left| v_{j+1} \right|
&= \left| r_j t_j + s_j v_j \right| \nonumber \\
&\le |r_j| |t_j| + |s_j| |v_j| \nonumber \\
&\le \alpha |\sigma|^{2 j - 2 k} \left( |\lambda|^j + \beta j |\sigma|^{-(k+j)} \right)
+ \left( |\sigma| + \alpha |\sigma|^{-k} \right) \beta j |\sigma|^{j - 2k} \nonumber \\
&\le \beta j |\sigma|^{(j+1) - 2 k}
+ \left( 2 \alpha \beta j |\sigma|^{-(k+1)} + \frac{\alpha}{|\sigma|} \right) |\sigma|^{(j+1) - 2 k} \nonumber \\
&\le \beta j |\sigma|^{(j+1) - 2 k} + \frac{2 \alpha}{|\sigma|} |\sigma|^{(j+1) - 2 k} \nonumber \\
&< \beta (j+1) |\sigma|^{(j+1) - 2 k}, \nonumber
\end{align}
and
\begin{align}
\left| w_{j+1} - \sigma^{(j+1)} \right|
&= \left| r_j u_j + s_j w_j - \sigma^{(j+1)} \right| \nonumber \\
&\le |r_j| |u_j| + \left| s_j - \sigma \right| \left| w_j - \sigma^{j} \right|
+ \left| s_j - \sigma \right| |\sigma|^{j} + |\sigma| \left| w_j - \sigma^{j} \right| \nonumber \\
&\le \beta j |\sigma|^{(j+1)-k}
+ \left( 2 \alpha \beta j |\sigma|^{-(k+1)} + \alpha |\sigma|^{-1} \right) |\sigma|^{(j+1)-k} \nonumber \\
&\le \beta j |\sigma|^{(j+1)-k}
+ \frac{2\alpha}{|\sigma|} |\sigma|^{(j+1)-k} \nonumber \\
&< \beta (j+1) |\sigma|^{(j+1)-k},
\end{align}
for sufficiently large values of $k$.
\end{proof}

\begin{lemma}
If $|\lambda \sigma| \leq 1$ then in a neighbourhood of $(x^*,0)$,
\begin{equation}
T_0^k(x,y) = \begin{bmatrix}
\lambda^k x \left( 1 + \cO(kxy) \right) \\
\sigma^k y \left(1 + \cO(kxy) \right)
\end{bmatrix}.
\label{eq:T0kl1}
\end{equation}
\label{le:T0k}
\end{lemma}
\begin{proof}
Write
\begin{equation}
    T_{0}^{k}(x,y) = \begin{bmatrix}
    \lambda^{k}x\left(1 + xyF_{k}(x,y)\right)\\
    \sigma^{k}y\left( 1 + xyG_{k}(x,y)\right)
    \end{bmatrix}.
    \label{eq:T0k1}
\end{equation}
Let $\delta = \frac{2^{\frac{1}{3}}-1}{2R}$. We will show that if $k|xy| \leq \delta$, then
\begin{equation}
    |F_{k}(x,y)| \leq 2kR, \hspace{0.2cm} |G_{k}(x,y)| \leq 2kR,
    \label{eq:FkGkl1}
\end{equation}
and this will complete the proof. We prove \eqref{eq:FkGkl1} by induction on $k$. Equation \eqref{eq:FkGkl1} is certainly true for $k=1$ because $|F_{1}(x,y)| = |F(x,y)| \leq R < 2R$, and similarly for $G_{1}$. Suppose \eqref{eq:FkGkl1} is true for some $k \geq 1$ (this is our induction hypothesis). By matching terms in $T_{0}^{k+1} = T_{0} \circ T_{0}^{k}$ we obtain  
\begin{equation}
    F_{k+1}(x,y) = F_{k}(x,y) + \lambda^{k}\sigma^{k}F\left(T_{0}^{k}(x,y)\right) \left( 1 + xyF_{k}(x,y)\right)^{2} \left( 1 + xyG_{k}(x,y)\right).
\end{equation}
By applying the induction hypothesis we obtain
\begin{equation}
    |F_{k+1}(x,y)| \leq 2kR + |F(T_{0}^{k}(x,y))|\left( 1 + |xy|2kR\right)^{3}.
\end{equation}
Since $k|xy|$ is small and $(x,y) \approx (x^*,0)$, we can assume $T_{0}^{k}(x,y) \in \cN$ and so $|F\left(T_{0}^{k}(x,y)\right)| \leq R$. Also $1 + |xy|2kR \leq 2^{\frac{1}{3}}$, so
\begin{equation}
    |F_{k+1}(x,y)| \leq 2(k+1)R,
\end{equation}
as required (and the result for $G_{k+1}$ is obtained similarly).
\end{proof}

\begin{proof}[Proof of Theorem \eqref{th:necessaryConditions}]
For brevity we only provide details in the case $|\lambda \sigma| \le 1$.
Since $T_0$ and $T_1$ are both locally invertible the case $|\lambda \sigma| > 1$
can be accommodated by considering $f^{-1}$ in place of $f$.

The trace of \eqref{eq:DfkPlusm} is
\begin{equation}
\tau_k = {\rm trace} \left( P_k D T_1 \left( x^{(k)}_k, y^{(k)}_k \right) \right),
\nonumber
\end{equation}
where $P_k$ is given by \eqref{eq:Pjl} with $j = k$, and so

\begin{equation}
\tau_k = {\rm trace} \left( \begin{bmatrix} t_k & u_k \\ v_k & w_k \end{bmatrix}
\begin{bmatrix} c_1 + \cO \left( x^{(k)}_k, y^{(k)}_k - y^* \right)
& c_2 + \cO \left( x^{(k)}_k, y^{(k)}_k - y^* \right) \\
d_1 + \cO \left( x^{(k)}_k, y^{(k)}_k - y^* \right)
& d_2 + \cO \left( x^{(k)}_k, y^{(k)}_k - y^* \right)
\end{bmatrix} \right).
\label{eq:tauk4l}
\end{equation}
From Lemma \ref{le:PjBoundsl} we see that the leading order term in $\tau_k$ is $w_{k}d_{2}$. In particular, if $d_2 \ne 0$ then $|\tau_{k}| \rightarrow \infty$ as $k \rightarrow \infty$. But ${\rm SR}_k$-solutions are assumed to be stable for arbitrarily large values of $k$, so this is only possible if $d_2 = 0$, equation \eqref{eq:tangencyCondition}.
To verify \eqref{eq:det1Condition} and \eqref{eq:globalResonanceCondition} we
use \eqref{eq:T11} and \eqref{eq:T0kl1} to solve for fixed points of $T_0^k \circ T_1$ in order to derive the leading order component of $\left(x_{k}^{(k)},y_{k}^{(k)}\right)$. Since $\left(x_{k}^{(k)},y_{k}^{(k)}\right) \rightarrow \left( 0, y^{*}\right)$ as $k \rightarrow \infty$, it is convenient to write
\begin{align}
    \left( x_{k}^{(k)},y_{k}^{(k)}\right) = \left( \lambda^{k}(x^{*} + \Phi_{k}), y^{*}+\Psi_{k} \right).
    \label{eq:fpcomposed}
\end{align}
Notice $\Psi_k \to 0$ because $y^{(k)}_k \to y^*$. 
Since $y^{(k)}_0 = \cO \left( \sigma^{-k} \right)$ by Lemma \ref{le:xjyjBound},
the first component of \eqref{eq:T0kl1} gives
$x^{(k)}_k = \lambda^k x^{(k)}_0 + \cO \left( k \lambda^{k} \sigma^{-k} \right)$.
Since $x^{(k)}_0 \to 0$ as $k \to \infty$
we can conclude that in \eqref{eq:fpcomposed} we also have $\Phi_k \to 0$.

Next we use \eqref{eq:T11} and \eqref{eq:T0kl1}
to obtain formulas for $\Phi_k$ and $\Psi_k$
based on the knowledge that $\Phi_k, \Psi_k \to 0$ as $k \to \infty$.
By substituting \eqref{eq:fpcomposed} into \eqref{eq:T11} we obtain
\begin{align}
x^{(k)}_0 &= x^* + \cO \left( \lambda^k, \Psi_k \right),
\label{eq:xk0l} \\
y^{(k)}_0 &=  d_{1} x^{*} \lambda^k + d_{1} \Phi_k \lambda^k
+ d_3 x^{*^2} \lambda^{2 k} + d_4 x^* \Psi_k \lambda^k  + d_5 \Psi_k^2
\nonumber \\ &\quad+ \cO \left( \lambda^{3 k}, \Phi_k \lambda^{2 k}, \Phi_k \Psi_k \lambda^k, \Psi_k^3 \right).
\label{eq:yk0l}
\end{align}
We then substitute \eqref{eq:xk0l} and \eqref{eq:yk0l} into \eqref{eq:T0kl1},
noting that $k xy = \cO \left( k \sigma^{-k} \right)$ by \eqref{eq:xjyjBound} with $j = 0$, to obtain for the $y$-component
\begin{align}
y^{(k)}_k &= d_{1}x^*\lambda^{k}\sigma^{k} + \sigma^{k} \Bigl[ d_{1} \Phi_k \lambda^k 
+ d_{3}x^{*^2}\lambda^{2k} + d_4 x^* \Psi_k \lambda^k + d_5 \Psi_k^2
\nonumber \\ &\quad+ \cO \left(k \sigma^{-k}, \lambda^{3 k}, \Phi_k \lambda^{2 k}, \Phi_k \Psi_k \lambda^k, \Psi_k^3 \right) \Bigr].
\label{eq:ykkl}
\end{align}
We now match the two expressions for $y^{(k)}_k$, \eqref{eq:fpcomposed} and \eqref{eq:ykkl}.  If $y^{*} \neq \lambda^{k}\sigma^{k}d_{1}x^{*}$ then because $d_5 \ne 0$ the $\sigma^k d_5 \Psi_k^2$ term in \eqref{eq:ykkl} must balance the $y^*$ term in \eqref{eq:fpcomposed}, and so we must have $\Psi_{k} \sim t \sigma^{-\frac{k}{2}}$ for some $t \neq 0$. In this case the $(2,2)$-element of $DT_{1}\left(x_{k}^{(k)},y_{k}^{(k)}\right)$ is asymptotic to $2d_{5}t \sigma^{-\frac{k}{2}}$.  Consequently from \eqref{eq:tauk4l} and Lemma \ref{le:PjBoundsl}, $\tau_k$ is asymptotic to $2 d_5 t \sigma^{\frac{k}{2}}$.  This diverges as $k \to \infty$ contradicting the stability assumption of the ${\rm SR}_k$-solutions.

Therefore we must have $y^{*} = d_{1} \lambda^{k}\sigma^{k} x^{*}$, which implies $|\lambda \sigma| = 1$ and $|d_{1}| = \frac{y^{*}}{x^{*}}$. Moreover, if $\lambda \sigma = 1$ then $d_1 = \frac{y^*}{x^*}$.
\end{proof}
\section{The orientation-preserving case}
\label{sec:positiveCase}
In this section we prove Theorems~\ref{th:a1b1} and \ref{th:sufficientConditionsPositiveCase}.


\begin{proof}[Proof of Theorem \ref{th:a1b1}]
We write
\begin{equation}
\left( x^{(k)}_k, y^{(k)}_k \right) = \left( \lambda^k (x^* + \Phi_k), y^* + \Psi_k \right),
\label{eq:xkyk4}
\end{equation}
where $\Phi_k, \Psi_k \to 0$ as $k \to \infty$ as shown in the proof of Theorem~\ref{th:necessaryConditions}.
By substituting \eqref{eq:xkyk4} into \eqref{eq:T11} with $d_1 = \frac{y^*}{x^*}$ and $d_2 = 0$ we obtain
\begin{align}
x^{(k)}_0 &= x^* + c_1 x^* \lambda^k + c_2 \Psi_k
+ \cO \left( \lambda^{2 k}, \lambda^k \Phi_k, \Psi_k^2 \right),
\label{eq:xk04} \\
y^{(k)}_0 &=  y^* \lambda^k + \tfrac{y^*}{x^*} \Phi_k  \lambda^k
+ d_3 x^{*^2} \lambda^{2 k} + d_4 x^* \lambda^k \Psi_k + d_5 \Psi_k^2
+ \cO \left( \lambda^{3 k}, \lambda^{2 k} \Phi_k, \lambda^k \Phi_k \Psi_k,\Psi_k^3 \right).
\label{eq:yk04}
\end{align}
Similar to Lemma \ref{le:T0k} it can be shown that
\begin{equation}
T_0^k(x,y) = \begin{bmatrix}
\lambda^k x \left( 1 + k a_1 x y + \cO \left( k^2 x^2 y^2 \right) \right) \\
\frac{1}{\lambda^k} y \left( 1 + k b_1 x y + \cO \left( k^2 x^2 y^2 \right) \right)
\end{bmatrix}.
\label{eq:T0k4}
\end{equation}
For brevity we omit a derivation which can be achieved by induction as in the proof of Lemma \ref{le:T0k}.
We then substitute \eqref{eq:xk04} and \eqref{eq:yk04} into \eqref{eq:T0k4},
noting that $k^2 x^2 y^2 = \cO \left( k^2 \lambda^{2 k} \right)$, to obtain
\begin{align}
x^{(k)}_k &=  x^* \lambda^k + a_1 x^{*^2} y^* k \lambda^{2 k}
+ \cO \left( \lambda^{2 k}, k \lambda^{2 k} \Phi_k, \lambda^k \Psi_k, k \lambda^k \Psi_k^2 \right),
\label{eq:xkk4} \\
y^{(k)}_k &= y^* + \lambda^{-k} \left[ \tfrac{y^*}{x^*} \lambda^k \Phi_k
+ b_1 x^* y^{*^2} k \lambda^{2 k} + d_5 \Psi_k^2
+ \cO \left( \lambda^{2 k}, k \lambda^{2 k} \Phi_k, \lambda^k \Psi_k, k \lambda^k \Psi_k^2, \Psi_k^3 \right) \right],
\label{eq:ykk4}
\end{align}
where we have only explicitly written the terms that will be important below.
By matching \eqref{eq:xkyk4} to \eqref{eq:ykk4} we obtain
\begin{equation}
\lambda^k \Psi_k = \tfrac{y^*}{x^*} \lambda^k \Phi_k
+ b_1 x^* y^{*^2} k \lambda^{2 k} + d_5 \Psi_k^2
+ \cO \left( \lambda^{2 k}, k \lambda^{2 k} \Phi_k, \lambda^k \Psi_k, k \lambda^k \Psi_k^2, \Psi_k^3 \right).
\label{eq:Psik4}
\end{equation}
By matching \eqref{eq:xkyk4} to \eqref{eq:xkk4} we obtain
a similar expression for $\lambda^k \Phi_k$ which we substitute into \eqref{eq:Psik4} to obtain
\begin{equation}
\lambda^k \Psi_k = (a_1 + b_1) x^* y^{*^2} k \lambda^{2 k} + d_5 \Psi_k^2
+ \cO \left( \lambda^{2 k}, k \lambda^{2 k} \Phi_k, \lambda^k \Psi_k, k \lambda^k \Psi_k^2, \Psi_k^3 \right).
\label{eq:Psik5}
\end{equation}
Notice that $\lambda^k \Psi_k \to 0$ faster than either $k \lambda^{2 k}$ or $\Psi_k^2$
(but possibly not both)
and by inspection the same is true for every error term in \eqref{eq:Psik5}.
Now suppose for a contradiction that $a_1 + b_1 \ne 0$.
Then since $d_5 \ne 0$ the
$k \lambda^{2 k}$ and $\Psi_k^2$ terms in \eqref{eq:Psik5} must balance.
Thus $\Psi_k \sim \zeta \sqrt{k} \lambda^k$, for some $\zeta \ne 0$, and~\eqref{eq:tauk4l} becomes
\begin{equation}
\tau_k = {\rm trace} \left( \begin{bmatrix} t_j & u_j \\ v_j & w_j \end{bmatrix}
\begin{bmatrix} c_1 + \cO \left( \sqrt{k} \lambda^k \right)
& c_2 + \cO \left( \sqrt{k} \lambda^k \right) \\
d_1 + \cO \left( \sqrt{k} \lambda^k \right)
& 2 d_5 \Psi_k + \cO \left( \lambda^k \right)
\end{bmatrix} \right).
\label{eq:tauk4}
\end{equation}
By Lemma \ref{le:PjBoundsl} with $j = k$
the term involving $\Psi_k$ provides the leading order contribution to \eqref{eq:tauk4},
specifically
\begin{equation}
\tau_k \sim 2 d_5 \zeta \sqrt{k}.
\label{eq:tauk}
\end{equation}
Thus $\tau_k \to \infty$ as $k \to \infty$
and so the ${\rm SR}_k$-solutions are unstable for sufficiently values of $k$.
This contradicts the stability assumption in the theorem statement,
therefore $a_1 + b_1 = 0$.
\end{proof}

\begin{proof}[Proof of Theorem \ref{th:sufficientConditionsPositiveCase}]
We look for ${\rm SR}_k$-solutions for which the $k^{\rm th}$ point has the form
\begin{equation}
\left( x^{(k)}_k, y^{(k)}_k \right) = \Big( \lambda^k \left( x^* + \phi_k \lambda^{k} + \cO \left( \lambda^{2 k} \right) \right),
y^* + \psi_k \lambda^k + \cO \left( \lambda^{2 k} \right) \Big),
\label{eq:xkyk3}
\end{equation}
where $\phi_k, \psi_k \in \mathbb{R}$. Recall $\left( x^{(k)}_k, y^{(k)}_k \right)$ is a fixed point of $T_0^k \circ T_1$, so is equal to its image under \eqref{eq:T11} and \eqref{eq:T0k4}. Through matching \eqref{eq:xkyk3} to this image we obtain
\begin{equation}
\begin{split}
\phi_k &= c_1 x^* + c_2 \psi_k + a_1 x^{*^2} y^* k, \\
\psi_k &= \frac{y^* \phi_k}{x^*} + d_3 x^{*^2} + d_4 x^* \psi_k + d_5 \psi_k^2 + b_1 x^* y^{*^2} k.
\end{split}
\label{eq:pkqk}
\end{equation}
By solving \eqref{eq:pkqk} simultaneously for $\phi_k$ and $\psi_k$ we find that the terms involving $k$ cancel because $a_1 + b_1 = 0$ and there are two solutions. The values of $\psi_k$ for these are given by
\begin{equation}
\psi_k^\pm = \frac{1}{2 d_5} \left( 1 - \frac{c_2 y^*}{x^*} - d_4 x^* \pm \sqrt{\Delta} \right).
\label{eq:qk}
\end{equation}
It is readily seen that these correspond to ${\rm SR}_k$-solutions (for some fixed $\eta > 0$)
for sufficiently large values of $k$ by Lemma \ref{le:xjyjBound}.

We now investigate the stability of the two solutions.
With \eqref{eq:xkyk3} equation~\eqref{eq:tauk4l} becomes
\begin{equation}
\tau_k = {\rm trace} \left(
\begin{bmatrix} t_k & u_k \\ v_k & w_k \end{bmatrix}
\begin{bmatrix} c_1 + \cO \left( \lambda^k \right) & c_2 + \cO \left( \lambda^k \right) \\
d_1 + \cO \left( \lambda^k \right) & d_4 x^* \lambda^k + 2 d_5 \psi_k \lambda^k + \cO \left( \lambda^{2 k} \right) \end{bmatrix} \right).
\label{eq:taukAgain}
\end{equation}
Thus by Lemma \ref{le:PjBoundsl} we have
\begin{align}
\lim_{k \to \infty} \tau_k = d_4 x^* + 2 d_5 \psi_k^\pm
= 1 - \frac{c_2 y^*}{x^*} \pm \sqrt{\Delta} \,.
\label{eq:tauk2}
\end{align}
Also the determinant of $D \left( T_0^k \circ T_1 \right)(x_k,y_k)$ converges to
\begin{align}
\lim_{k \to \infty} \delta_k = c_1 d_2 - c_2 d_1
= -\frac{c_2 y^*}{x^*}.
\label{eq:deltak2}
\end{align}
To show that $\psi_k^-$ generates an asymptotically stable ${\rm SR}_k$-solution we verify
(i) $\delta_k - \tau_k + 1 > 0$,
(ii) $\delta_k + \tau_k + 1 > 0$, and 
(iii) $\delta_k < 1$,
for sufficiently large values of $k$ (see Fig.~\ref{fig:stability_triangle}).
From \eqref{eq:tauk2} and \eqref{eq:deltak2} with $\psi_k^-$ we have
\begin{align}
\lim_{k \to \infty} (\delta_k - \tau_k + 1) &= \sqrt{\Delta}, \label{eq:sqrtDelta} \\
\lim_{k \to \infty} (\delta_k + \tau_k + 1) &= 2 - \frac{2 c_2 y^*}{x^*} - \sqrt{\Delta}, \\
\lim_{k \to \infty} (1 - \delta_k) &= 1 + \frac{c_2 y^*}{x^*}.
\end{align}
These limits are all positive by \eqref{eq:stabilityCondition},
hence conditions (i)--(iii) (given just above equation \eqref{eq:sqrtDelta})
are satisfied for sufficiently large values of $k$.
Finally observe that with instead $\psi^+_k$ we have
$\lim_{k \to \infty} (\delta_k - \tau_k + 1) = -\sqrt{\Delta} < 0$
and $\lim_{k \to \infty} |\delta_k| = \frac{|c_2| y^*}{x^*} < 1$ hence $\psi^+_k$ generates a saddle ${\rm SR}_k$-solution.
\end{proof}

\section{The orientation-reversing case}
\label{sec:negativeCase}



\begin{proof}[Proof of Theorem \ref{th:sufficientConditionsNegativeCase}]
As in the proof of Theorem \ref{th:sufficientConditionsPositiveCase} we assume $\left( x^{(k)}_k, y^{(k)}_k \right)$ has the form \eqref{eq:xkyk3}.  This point is a fixed point of $T_0^k \circ T_1$ where
$T_0^k$ again has the form \eqref{eq:T0k4} except now $a_1 = b_1 = 0$.
By composing this with \eqref{eq:T11} we obtain
\begin{equation}
\left( T_0^k \circ T_1 \right)\left( x^{(k)}_k, y^{(k)}_k \right) =
\begin{bmatrix}
x^* \lambda^k + \left( c_1 x^* + c_2 \psi_k \right) \lambda^{2 k} + \cO \left( \lambda^{3 k} \right) \\
(-1)^k d_1 x^* + (-1)^k \left( d_1 \phi_k + d_3 x^{*^2} + d_4 x^* \psi_k + d_5 \psi_k^2 \right) \lambda^k + \cO \left( \lambda^{2 k} \right)
\end{bmatrix},
\label{eq:T0kT1xkyk}
\end{equation}
where we have substituted $\lambda \sigma = -1$ and $d_2 = 0$.
For the remainder of the proof we assume $k$ is even in the case $d_1 = \frac{y^*}{x^*}$
and $k$ is odd in the case $d_1 = -\frac{y^*}{x^*}$.
In either case
\begin{equation}
(-1)^k d_1 = \frac{y^*}{x^*},
\label{eq:d1Parity}
\end{equation}
and so the leading-order terms of \eqref{eq:xkyk3} and \eqref{eq:T0kT1xkyk} are the same.
By matching the next order terms we obtain
\begin{equation}
\begin{split}
\phi_k &= c_1 x^* + c_2 \psi_k \,, \\
\psi_k &= \frac{y^* \phi_k}{x^*} + (-1)^k \left( d_3 x^{*^2} + d_4 x^* \psi_k + d_5 \psi_k^2 \right).
\end{split}
\label{eq:pkqkNegativeCase}
\end{equation}
These produce the following two solutions for the value of $\psi_k$
\begin{equation}
\psi_k^\pm = \frac{(-1)^k}{2 d_5} \left( 1 - \frac{c_2 y^*}{x^*} - \frac{d_4 y^*}{d_1} \pm \sqrt{\Delta} \right),
\label{eq:qkNegativeCase}
\end{equation}
where we have further used \eqref{eq:d1Parity}.
Analogous to the proof of Theorem \ref{th:sufficientConditionsPositiveCase} we obtain
\begin{align}
\lim_{k \to \infty} \tau_k &= (-1)^k \left( d_4 x^* + 2 d_5 \psi_k^\pm \right) 
\label{eq:limtau}\\
&= 1 - \frac{c_2 y^*}{x^*} \pm \sqrt{\Delta},
\nonumber \\
\lim_{k \to \infty} \delta_k &= -\frac{c_2 y^*}{x^*},
\nonumber
\end{align}
and the proof is completed via the same stability arguments.
\end{proof}
\section{Conclusions}
\label{sec:conc}
In this paper we have considered single-round periodic solutions
associated with homoclinic tangencies of two-dimensional $C^\infty$ maps.
We have formalised these as ${\rm SR}_k$-solutions via Definition \ref{df:SRkSolution}.
The key arguments leading to our results are centred around calculations of $\tau_k$ ---
the sum of the eigenvalues associated with an ${\rm SR}_k$-solution.
Immediately we see from Fig.~\ref{fig:stability_triangle} that if $|\tau_k| > 2$ then the ${\rm SR}_k$-solution is unstable.

We first showed that if conditions \eqref{eq:tangencyCondition}--\eqref{eq:globalResonanceCondition} do not all hold then $|\tau_k| \to \infty$ as $k \to \infty$,
thus at most finitely many ${\rm SR}_k$-solutions can be stable (Theorem \ref{th:necessaryConditions}).
Equation \eqref{eq:det1Condition}, namely $|\lambda \sigma| = 1$, splits into two fundamentally distinct cases.
If $\lambda \sigma = -1$
then the resonant terms in $T_0$ that cannot be eliminated by a coordinate change
are of sufficiently high order that they have no bearing on the results.
In this case $\tau_k$ converges to the finite value \eqref{eq:limtau} and
infinitely many ${\rm SR}_k$-solutions can indeed be stable
as a codimension-three phenomenon (Theorem \ref{th:sufficientConditionsNegativeCase}).
If instead $\lambda \sigma = 1$ then
$|\tau_k|$ is asymptotically proportional to $\sqrt{k}$ unless $a_1 + b_1 = 0$,
i.e.~the coefficients of the leading-order resonance terms cancel (Theorem \ref{th:a1b1}).
This is the only additional condition needed to have infinitely many ${\rm SR}_k$-solutions,
aside from the inequalities $\Delta > 0$ and \eqref{eq:stabilityCondition},
thus in this case the infinite coexistence is codimension-four (Theorem \ref{th:sufficientConditionsPositiveCase}).

We have demonstrated the infinite coexistence with an artificial family of maps in \S\ref{sec:example}.
It remains to identify the infinite coexistence in previously studied prototypical maps.
The relatively high codimension possibly explains why this does not appear to have been done already.
Yet the codimension is not so high that the phenomenon cannot be expected to
have an important role in some applications.
A determination of the typical bifurcation structure that surrounds
these codimension-three and four bifurcations remains for future work.
It also may be of interest to generalise the results to incorporate cubic and higher order tangencies
and extend the results to higher dimensional maps.

\bibliographystyle{plain}

\end{document}